\documentclass[12pt]{article}
\usepackage{amsmath}
\usepackage{amssymb}
\usepackage{amsthm}
\usepackage{tikz,url}
\usepackage[backend=biber]{biblatex}
\usepackage[hidelinks]{hyperref}

\newtheorem{theorem}{Theorem}[section]
\newtheorem{lemma}[theorem]{Lemma}

\numberwithin{equation}{section}

\DeclareMathOperator{\dist}{dist}
\DeclareMathOperator{\sign}{sign}

\newcommand{\R}{\mathbb{R}}

\newcommand{\Z}{\mathbb{Z}}

\title{Asymptotic Spectral Bounds under Odd-Girth Constraints}
\author{Chen Wang}

\begin{document}

\maketitle

\begin{abstract}
    Odd girth and the sum of the largest and smallest adjacency eigenvalues are two common indicators of how close a graph is to being bipartite. This paper investigates the asymptotic behavior of this eigenvalue sum under odd-girth constraints. We study the associated abstract moment problem over finite real lists satisfying the vanishing odd-moment identities imposed by odd girth at least $g$ and the standard second-moment bound for graph spectra. We prove that the supremum of $(\lambda_1+\lambda_n)/n$ over these lists lies between $(1-o(1))(\log g)^2/g^3$ and $(10/7+o(1))(\log g)^2/g^3$ as $g\to\infty$ through the odd integers. The upper bound applies in particular to graphs of odd girth at least $g$.
    
    For regular graphs of odd girth $g\geq5$, we obtain a stronger bound $(\lambda_1+\lambda_n)/n \leq 12/((g-1)(g-2)(g-3))$, which is optimal up to a constant factor.
\end{abstract}

\section{Introduction}

A connected graph is bipartite if and only if its largest and smallest adjacency eigenvalues sum to zero. This characterization suggests a quantitative question: how small must this sum be in a graph containing no short odd cycles?

For a graph $G$ on $n$ vertices, write $\lambda_1(G)\geq\cdots\geq\lambda_n(G)$ for its adjacency eigenvalues. The odd girth of $G$ is the length of its shortest odd cycle, taken to be infinite when $G$ is bipartite. For an odd integer $g\geq5$, let
\[
    \gamma_g = \sup_G \frac{\lambda_1(G)+\lambda_n(G)}{|V(G)|},
\]
where the supremum ranges over all finite simple graphs of odd girth at least $g$. The odd cycle $C_g$ gives
\[
    \gamma_g \geq \frac{2-2\cos(\pi/g)}{g} = \frac{\pi^2+o(1)}{g^3}.
\]
The case $g=5$ corresponds to triangle-free graphs. Brandt \cite{brandt1998local} proved that $(\lambda_1+\lambda_n)/n\leq3-2\sqrt{2}$ for regular triangle-free graphs, and Csikv\'ari \cite{csikvari2022sum} extended this bound to all triangle-free graphs.

Related work concerns the least signless Laplacian eigenvalue, which equals $\lambda_1+\lambda_n$ for regular graphs. Balogh, Clemen, Lidick\'y, Norin, and Volec \cite{balogh2023spectrum} proved that this eigenvalue is at most $15n/94$ for every triangle-free graph on $n$ vertices. More recently, Zhou \cite{zhou2026signless} established the bound $0.0569n$ for graphs of odd girth at least seven, also without assuming regularity. Abiad, Kumar, and Pragada \cite{abiad-kumar-pragada2026} raised the corresponding question of bounding the least signless Laplacian eigenvalue for graphs of high odd girth.

For growing odd girth, Abiad, Taranchuk, and van Veluw \cite{abiad-taranchuk-vanveluw} proved that $\gamma_g=O(g^{-1})$; they also obtained the stronger bound $\gamma_7<0.0396$ for the case $g=7$. Yip \cite{yip} subsequently improved the asymptotic bound to $\gamma_g=O((\log g)^3/g^3)$, leaving a polylogarithmic gap between the upper and lower bounds.

Yip's argument also raises a question about the information needed to obtain such bounds. If $G$ has odd girth at least $g$, then its adjacency eigenvalues satisfy
\[
    \sum_{i=1}^n \lambda_i^j=0 \qquad (j=1,3,\ldots,g-2),
\]
since $G$ has no closed walks of these lengths. They also satisfy
\[
    \sum_{i=1}^n \lambda_i^2 =2|E(G)|\leq n\lambda_1.
\]
These are the only properties of graph spectra used in his proof. Accordingly, Yip defined $\gamma'_g$ as the supremum of $(\lambda_1+\lambda_n)/n$ over all finite real lists $\lambda_1\geq\cdots\geq\lambda_n$ satisfying these conditions, and established
\[
    \gamma_g\leq\gamma'_g =O\left(\frac{(\log g)^3}{g^3}\right).
\]
He asked for the asymptotic order of $\gamma'_g$.

The abstract moment problem admits a scale-invariant formulation. For a nonzero real list
$\lambda=(\lambda_1,\ldots,\lambda_n)$, ordered decreasingly, define
\[
    \Phi(\lambda) =\frac{\lambda_1(\lambda_1+\lambda_n)}{\sum_{i=1}^n\lambda_i^2}.
\]
This definition extends naturally to unsorted tuples and multisets by replacing $\lambda_1$ and $\lambda_n$ with the maximum and minimum elements, respectively. Positive scaling preserves both $\Phi$ and the vanishing odd moments. By rescaling a list so that
$\sum_i\lambda_i^2=n\lambda_1$, we obtain the equivalent formulation
\[
    \gamma'_g=\sup_\lambda \Phi(\lambda),
\]
where the supremum ranges over all nonzero finite real lists satisfying
\[
    \sum_i\lambda_i^j=0 \qquad (j=1,3,\ldots,g-2).
\]
Thus the quadratic constraint can be absorbed into the objective, leaving an optimization problem governed by vanishing odd moments.

Our first main result determines the asymptotic order of the abstract moment problem. More precisely, we prove that
\[
    (1-o(1))\frac{(\log g)^2}{g^3} \leq \gamma'_g \leq \left(\frac{10}{7}+o(1)\right)\frac{(\log g)^2}{g^3}
\]
as $g\to\infty$ through odd integers.
This improves Yip's upper bound by a factor of $\log g$ and shows that the remaining logarithmic factor is unavoidable under the moment and quadratic constraints alone.

Our second main result concerns unweighted regular graphs. For a graph $G$ with at least one edge, write $\Phi(G)$ for $\Phi$ evaluated
on its adjacency spectrum. Then 
\[
    \frac{\lambda_1(G)+\lambda_n(G)}{n}\leq \Phi(G),
\]
with equality whenever $G$ is regular. Indeed, for a $d$-regular graph, $\lambda_1(G)=d$ and $\sum_i\lambda_i(G)^2=nd$, so
\[
    \Phi(G)=\frac{d+\lambda_n(G)}{n}.
\]
We prove that every $d$-regular graph of odd girth exactly $g\geq5$ satisfies
\[
    \frac{d+\lambda_n(G)}{n} \leq \frac{12}{(g-3)(g-2)(g-1)}.
\]
Together with the example of the odd cycle, this determines the asymptotic order for regular graphs:
\[
    \gamma_g^{\mathrm{reg}}=\Theta(g^{-3}),
\]
where $\gamma_g^{\mathrm{reg}}$ denotes the supremum defining $\gamma_g$ restricted to regular graphs.

Consequently,
\[
    \frac{\gamma'_g}{\gamma_g^{\mathrm{reg}}} =\Theta((\log g)^2).
\]
The moment relaxation therefore permits values larger by a logarithmic factor than those attainable by regular graphs, even though the two objectives agree on every regular graph. Our results leave the asymptotic order of $\gamma_g$ for general graphs undetermined.

The scale-invariant formulation also extends naturally to \emph{edge-weighted graphs}. Let $G$ be a finite simple graph with at least one edge and positive edge weights $w_e$, and let $A$ be its weighted adjacency matrix. Then
\[
    \Phi(A) =\frac{\lambda_1(A)(\lambda_1(A)+\lambda_n(A))}{\operatorname{tr}(A^2)}=\frac{\lambda_1(A)(\lambda_1(A)+\lambda_n(A))}{2\sum_{e\in E(G)}w_e^2}.
\]
This quantity is unchanged when all edge weights are multiplied by a common positive constant. If the underlying graph has odd girth at least $g$, its weighted adjacency spectrum still satisfies
\[
    \operatorname{tr}(A^j)=0 \qquad (j=1,3,\ldots,g-2).
\]
Consequently, $\Phi(A)\leq \gamma'_g$, so every upper bound for the abstract moment problem applies to arbitrarily weighted graphs of the same odd girth.

This paper is organized as follows. In Section~\ref{sec:preliminaries}, we introduce the necessary background and notation. Sections~\ref{sec:moment_problem} and \ref{sec:lower_bounds} analyze the abstract moment problem and establish the asymptotic order of $\gamma'_g$. Section~\ref{sec:regular_graphs} focuses on unweighted regular graphs and proves the corresponding bounds for $\gamma_g^{\mathrm{reg}}$. Finally, Section~\ref{sec:conclusion} discusses the implications of our results and outlines directions for future research.

\section{Preliminaries}
\label{sec:preliminaries}

Throughout the paper, $\log$ denotes the natural logarithm unless a base is indicated.

In this section, we collect the polynomial identities and estimates used in the upper- and lower-bound arguments for the abstract moment problem. We begin with Lagrange interpolation, followed by the properties of Chebyshev and Legendre polynomials needed below. For the standard facts about orthogonal polynomials, we refer to \cite{Szego1975}.

\subsection{Lagrange interpolation}
We record the following Lagrange interpolation identities, which will be used throughout the upper-bound argument.

\begin{lemma}[Lagrange interpolation]\label{lem:lagrange_interpolation}
    Let $\xi_1,\ldots,\xi_n$ be distinct real numbers, where $n\geq1$, and put
    \[
        V(z):=\prod_{j=1}^n(z-\xi_j),\qquad L_j(z):=\frac{V(z)}{(z-\xi_j)V'(\xi_j)}.
    \]
    Every polynomial $f$ of degree at most $n-1$ satisfies
    \begin{equation}\label{eq:lagrange_interpolation}
        f(z)=\sum_{j=1}^n f(\xi_j)L_j(z).
    \end{equation}
    In particular, if $\deg f\leq n-2$, then
    \begin{equation}\label{eq:lagrange_cancellation}
        \sum_{j=1}^n\frac{f(\xi_j)}{V'(\xi_j)}=0.
    \end{equation}
    Moreover, the derivatives of the interpolation polynomials satisfy
    \begin{equation}\label{eq:lagrange_derivatives}
        L_j'(\xi_i)=
        \begin{cases}
            \displaystyle\sum_{\ell\ne i}\frac{1}{\xi_i-\xi_\ell}, &j=i,\\[6pt]
            \displaystyle\frac{V'(\xi_i)}{(\xi_i-\xi_j)V'(\xi_j)}, &j\ne i,
        \end{cases}
        \qquad i,j=1,\ldots,n.
    \end{equation}
\end{lemma}
\begin{proof}
    Both sides of \eqref{eq:lagrange_interpolation} have degree at most $n-1$ and agree at the $n$ distinct nodes $\xi_1,\ldots,\xi_n$, so they are equal. Comparing the coefficients of $z^{n-1}$ gives \eqref{eq:lagrange_cancellation} when $\deg f\leq n-2$. Finally, \eqref{eq:lagrange_derivatives} follows by differentiating
    \[
        L_j(z)=\prod_{\ell\ne j}\frac{z-\xi_\ell}{\xi_j-\xi_\ell}
    \]
    and evaluating at $z=\xi_i$.
\end{proof}

\subsection{Chebyshev polynomials}
We record the properties of Chebyshev polynomials needed in the upper- and lower-bound arguments.

\begin{lemma}[Chebyshev polynomials]\label{lem:chebyshev_polynomials}
    The Chebyshev polynomials of the first kind are defined by
    \[
        T_0(x)=1,\qquad T_1(x)=x,\qquad
        T_{n+1}(x)=2xT_n(x)-T_{n-1}(x)\quad(n\geq1).
    \]
    For every integer $n\geq0$, the polynomial $T_n$ has degree $n$ and satisfies $T_n(-x)=(-1)^nT_n(x)$. Its leading coefficient is $2^{n-1}$ when $n\geq1$. Moreover,
    \begin{align}
        T_n(\cos\theta)&=\cos(n\theta),
            &&\theta\in\R,\label{eq:chebyshev_cosine}\\
        T_n(\cosh u)&=\cosh(nu),
            &&u\geq0.\label{eq:chebyshev_hyperbolic}
    \end{align}
    In particular,
    \begin{equation}\label{eq:chebyshev_basic_properties}
        |T_n(x)|\leq1\quad(|x|\leq1),\qquad
        T_n(1)=1,\qquad T_n'(1)=n^2,
    \end{equation}
    and $T_n$ is nondecreasing on $[1,\infty)$, with $T_n(x)\geq1$ there. For $0<\theta<\pi$, we have
    \begin{equation}\label{eq:chebyshev_derivative}
        T_n'(\cos\theta)=\frac{n\sin(n\theta)}{\sin\theta}.
    \end{equation}
    When $n\geq1$, the zeros of $T_n$ are simple and are given by
    \begin{equation}\label{eq:chebyshev_zeros}
        \cos\left(\frac{(2r-1)\pi}{2n}\right),
        \qquad r=1,\ldots,n.
    \end{equation}
\end{lemma}

\subsection{Legendre polynomials and reproducing kernels}
\label{sec:legendre_preliminaries}
We record the properties of Legendre polynomials and their reproducing kernels needed for the lower-bound argument.

\begin{itemize}
    \item The Legendre polynomials $\{P_k\}_{k\geq0}$, normalized such that $P_k(1)=1$, are orthogonal with respect to the Lebesgue measure on $[-1,1]$:
    \begin{equation}\label{eq:legendre_orthogonality}
        \int_{-1}^1 P_k(x) P_\ell(x) \, dx = \frac{2}{2k+1} \delta_{k\ell}.
    \end{equation}
    \item The Legendre polynomials satisfy the upper bounds
    \begin{align}
        |P_k(x)| &\leq 1, \quad x\in[-1,1], \label{eq:legendre_bound_1}\\
        1\leq P_k(\cosh \alpha) &\leq e^{k\alpha}, \quad \alpha\geq 0, \label{eq:legendre_bound_2}
    \end{align}
    the parity relation
    \begin{equation}\label{eq:legendre_parity}
        P_k(-x) = (-1)^k P_k(x),
    \end{equation}
    and the integral relation
    \begin{equation}\label{eq:legendre_integral}
        \int P_k(x)\,dx = \frac{P_{k+1}(x)-P_{k-1}(x)}{2k+1}+C, \quad k\geq 1.
    \end{equation}
    \item The reproducing kernel for the space of polynomials of degree at most $D$ is given by
    \begin{equation}\label{eq:legendre_kernel}
        K_D(x,y) = \sum_{k=0}^D \frac{2k+1}{2} P_k(x) P_k(y),
    \end{equation}
    which satisfies the reproducing property
    \begin{equation}\label{eq:legendre_reproducing}
        \int_{-1}^1 K_D(x,y) p(y) \, dy = p(x)
    \end{equation}
    for all polynomials $p$ of degree at most $D$ and all $x\in\R$.
\end{itemize}
We also make the convention that $P_{-1}(x) \equiv 1$ so that the integral relation \eqref{eq:legendre_integral} holds for $k=0$ as well.

\section{Upper Bounds for the Abstract Moment Problem}
\label{sec:moment_problem}
In this section, we prove the following upper bound for the abstract moment problem.

\begin{theorem}\label{thm:moment_upper_bound}
    As $g\to\infty$ through odd integers,
    \[
        \gamma'_g \leq (10/7+o(1))\frac{(\log g)^2}{g^3}.
    \]
\end{theorem}

We first observe that dropping the first-moment constraint does not change the supremum. Clearly, dropping this constraint enlarges the feasible set and hence cannot decrease the supremum. Conversely, let $X$ be a feasible multiset for the problem without the first-moment constraint, with positive objective, normalized to have minimum $-1$ and
maximum $a>1$. For any integer $m>a$, form a new multiset $Y$ by adjoining to $X$ exactly $m$ copies of the multiset $-X/m$. For every odd positive integer $j$, we have
\[
    \sum_{u\in Y}u^j=\left(1-m^{1-j}\right)\sum_{u\in X}u^j.
\]
Thus $Y$ has vanishing first moment and retains all the required higher vanishing odd moments. Its minimum and maximum remain $-1$ and $a$, while
\[
    \sum_{u\in Y}u^2=\left(1+\frac1m\right)\sum_{u\in X}u^2.
\]
Consequently,
\[
    \Phi(Y)=\frac{m}{m+1}\Phi(X)\longrightarrow\Phi(X)\qquad\text{ as }m\to\infty.
\]
This proves that the two suprema coincide. For the remainder of this section, we work with the abstract moment problem without the first-moment constraint.

Fix a feasible multiset of size $N$ with positive objective. By positive scaling, we may first impose the normalization
\[
    \sum_{i=1}^N x_i^2=1.
\]
The retained moment constraints define a closed subset of the unit sphere in $\R^N$, so the normalized feasible set is compact. Since $\Phi$ is continuous on this set, it attains a maximum. This maximum is at least the objective value of the original multiset and is therefore positive.

Let $x$ be such a maximizer. The vanishing third moment and the normalization $\sum_i x_i^2=1$ imply that $x$ has both positive and negative entries. We may therefore rescale $x$ to have minimum $-1$; its maximum is then $a>1$, since its objective is positive. The rescaled tuple remains a global maximizer among feasible $N$-tuples with minimum $-1$. Indeed, any competing feasible $N$-tuple with minimum $-1$ can be rescaled back onto the unit sphere without changing its objective, and therefore cannot have a larger objective value. We use this minimum normalization in the subsequent local arguments.

Every positive-objective feasible multiset is therefore dominated by a maximizing multiset of the same size. We will bound these maxima uniformly in $N$.

In this section, we use \emph{entry} for an individual coordinate, counting repetitions, and \emph{support location} for a distinct value occurring in the tuple. The multiplicity of a support location is the number of entries equal to it; a value is \emph{simple} if its multiplicity is one. The \emph{nonzero support} is the set of nonzero support locations.

Under the normalization above, entries equal to $a$ are called \emph{top entries}, and entries equal to $-1$ are called \emph{minimum entries}. An entry is \emph{residual} if it is strictly less than $a$, and \emph{interior} if it lies in $(-1,a)$. Thus zero entries, when present, are both residual and interior. The \emph{residual support} consists of the nonzero support locations other than $a$; it includes $-1$. Its locations in $(-1,a)$ form the \emph{interior residual support}. Once the largest entry is proved simple, we refer to it as the top entry.

The proof of the upper bound is divided into several steps. First, we derive local optimality conditions for a maximizer, and then use these conditions to conclude any such maximizer must have exactly $(g-1)/2$ distinct nonzero values. A sign-change argument for a signed tail-counting function then shows that every residual support location has multiplicity at least two, allowing us to impose splitting stability inequalities. These constraints ultimately lead to an asymptotic lower bound for the second moment $\sum_i x_i^2$. Finally, an estimate using Chebyshev polynomials controls the largest entry in terms of the second moment and yields the upper bound for the objective $\Phi$.

\subsection{Local optimality conditions}

Put $k:=(g-3)/2$, so that the retained moment conditions are
\[
    M_{2j+3}(x):=\sum_{i=1}^N x_i^{2j+3}=0, \qquad j=0,\ldots,k-1.
\]
We work with individual, unordered coordinates. In particular, equal entries are allowed to separate; their multiplicities are not fixed during local variations. The size $N$ remains fixed throughout the argument.

We first justify the local optimality conditions, including their use at the boundary where one or more entries equal $-1$.

\begin{lemma}\label{lem:upper_local_optimality}
    Let $x$ be a local maximizer of $\Phi$ among feasible $N$-tuples with minimum $-1$ and maximum at least $1$, and suppose that $\Phi(x)>0$. Then its largest entry is simple. Relabel it as $x_1=a>1$, and put
    \[
        \begin{gathered}
            Q:=\sum_{i=1}^N x_i^2,\qquad c:=\Phi(x)=\frac{a(a-1)}{Q},\\
            I_-:=\{i:x_i=-1\},\qquad b:=|I_- |.
        \end{gathered}
    \]
    There exists a polynomial $p$ in the span of $z^3,z^5,\ldots,z^{g-2}$ such that, for
    \[
        R(z):=2cz-p'(z),
    \]
    we have
    \begin{equation}\label{eq:upper_stationarity}
        R(a)=2a-1,\qquad
        R(x_i)=0\quad(-1<x_i<a),\qquad
        R(-1)=\frac{a}{b}.
    \end{equation}
    Moreover, every vector $v\in\R^N$ satisfying
    \begin{equation}\label{eq:upper_tangent_conditions}
        v_i=0\quad(i\in I_-),\qquad \sum_{i=1}^N(2\ell+1)x_i^{2\ell}v_i=0 \quad(\ell=1,\ldots,k)
    \end{equation}
    satisfies
    \begin{equation}\label{eq:upper_second_variation}
        \sum_{i=1}^N R'(x_i)v_i^2-2v_1^2\geq0.
    \end{equation}
\end{lemma}
\begin{proof}
    We first show that the moment equations remain regular when every minimum entry is held fixed. In fact, at any feasible tuple with maximum $a>1$, the entries strictly below the maximum have at least $k$ distinct nonzero absolute values. Otherwise, writing those absolute values as $u_1,\ldots,u_s$ with $s<k$, the polynomial
    \[
        f(z):=z^3\prod_{j=1}^s(z^2-u_j^2)
    \]
    belongs to the span of the constrained odd monomials. It vanishes at every entry strictly below the maximum, including every zero, and is strictly positive at every largest entry. This contradicts $\sum_i f(x_i)=0$.

    After removing the magnitude $1$, at least $k-1$ of these magnitudes have representatives strictly between $-1$ and $a$. Together with one largest entry, they supply $k$ free coordinates with distinct
    nonzero squares. On these coordinates, the moment Jacobian has entries
    \[
        (2\ell+1)x_i^{2\ell},\qquad \ell=1,\ldots,k.
    \]
    Dividing its rows by $2\ell+1$ and its columns by $x_i^2$ gives a Vandermonde matrix in the distinct squares. By Lemma~\ref{lem:lagrange_interpolation}, evaluation at these $k$ squares uniquely determines a polynomial of degree at most $k-1$, so this matrix is invertible.

    The implicit function theorem now shows that the moment equations, with all coordinates in $I_-$ fixed at $-1$, define a smooth manifold near $x$. Every vector satisfying \eqref{eq:upper_tangent_conditions} is the initial velocity of a smooth curve on this manifold. All other coordinates remain strictly greater than $-1$ in a sufficiently small neighborhood, and the maximum remains greater than $1$. Hence these curves are feasible for the original normalized problem. In particular, the required curves exist for both signs of their parameter.

    Suppose next that two largest entries are equal to $a$. Give them velocities $1$ and $-1$, and give every other entry velocity zero. This vector satisfies \eqref{eq:upper_tangent_conditions},
    so it extends to an exactly feasible curve $x(\varepsilon)=x+\varepsilon v+O(\varepsilon^2)$ with all
    minimum entries fixed. Along this curve the quadratic sum is $Q+O(\varepsilon^2)$, while for $\varepsilon>0$ the maximum is $a+\varepsilon+O(\varepsilon^2)$. Consequently,
    \[
        \Phi(x(\varepsilon))=\Phi(x)+\frac{2a-1}{Q}\varepsilon+O(\varepsilon^2),
    \]
    contradicting local maximality. Thus the largest entry is simple. After relabeling it as $x_1$, the maximum is represented by this same coordinate throughout a sufficiently small neighborhood,
    and the objective is smooth there.

    With $c=\Phi(x)$ fixed, the function
    \[
        z_1(z_1-1)-c\sum_i z_i^2=\left(\sum_i z_i^2\right)(\Phi(z)-c)
    \]
    has a local maximum on the moment manifold. By the regularity proved above, Lagrange multipliers give a polynomial $p$ in the stated span for which $R(a)=2a-1$ and $R(x_i)=0$ at all interior entries. 
    
    No interior stationarity equation is imposed at $-1$. Instead, the moment identities imply
    \[
        \sum_i x_i p'(x_i)=0.
    \]
    Therefore,
    \[
        2cQ=\sum_i x_iR(x_i) =a(2a-1)-bR(-1).
    \]
    Since $cQ=a(a-1)$, this gives $R(-1)=a/b$.

    This identity also describes variations that leave the face on which all minimum entries are fixed. If $z(t)$, $t\geq0$, is any differentiable feasible curve with $z(0)=x$ and $z'(0)=v$, then $v_i\geq0$ for $i\in I_-$. Differentiating the moments and using \eqref{eq:upper_stationarity}, we obtain
    \begin{equation}\label{eq:upper_boundary_first_variation}
        Q\left.\frac{d}{dt}\Phi(z(t))\right|_{t=0^+}
        =(2a-1)v_1-\sum_i R(x_i)v_i = -\frac{a}{b}\sum_{i\in I_-}v_i\leq0.
    \end{equation}
    Equality holds only if every minimum entry has zero initial velocity. Thus any first-order direction moving a minimum entry upward strictly decreases the objective.

    Finally, let $v$ satisfy \eqref{eq:upper_tangent_conditions}, and choose a smooth feasible curve with this initial velocity that keeps all minimum entries fixed for the whole curve. The first variation vanishes. Differentiating a second time, the acceleration terms cancel by stationarity and the
    differentiated moment equations, giving
    \[
        Q\left.\frac{d^2}{dt^2}\Phi(z(t))\right|_{t=0}=2v_1^2-\sum_iR'(x_i)v_i^2.
    \]
    Local maximality proves \eqref{eq:upper_second_variation}.
\end{proof}

In particular, if an interior support location $r\in(-1,a)$ has multiplicity at least two, we may give two copies velocities $1$ and $-1$. The tangent conditions hold, so
\begin{equation}\label{eq:upper_interior_splitting}
    R'(r)\geq0.
\end{equation}

Since $p$ has no linear term, $R(0)=0$ and $R'(0)=2c>0$. Zero entries therefore cause no loss of regularity and can remain in the fixed-size tuple throughout the local arguments. They do not contribute to the moments or the quadratic sum, and will not be counted as nonzero support locations.

Since $p$ contains only odd powers of degrees at least three, we may write
\begin{equation}\label{eq:R_form}
    R(z)=2c\left(z-h(z^2)\right), \qquad \deg h\leq k,\qquad h(0)=0.
\end{equation}
In particular, $R(z)-R(-z)=4cz$. Thus two opposite nonzero values cannot both occur as interior entries, since stationarity would make both values of $R$ zero. Moreover,
\[
    R(1)=R(-1)+4c=\frac{a}{b}+4c>0,
\]
so $1$ cannot occur as an entry either. Since the minimum is $-1$ and the maximum is $a>1$, it follows that distinct nonzero values occurring in the multiset have distinct squares.

\subsection{Reduction to minimal support}
We continue to write $x_1,\ldots,x_N$ for the individual entries, with $x_1=a$. Let $r_1,\ldots,r_s$ be the residual support locations. Write $m_j$ for the multiplicity of $r_j$, and put $y_j:=r_j^2$, its squared magnitude. We reserve index $0$ for the top and index $s$ for the minimum:
\[
    r_0=a,\qquad m_0=1,\qquad y_0=a^2, \qquad r_s=-1,\qquad m_s=b,\qquad y_s=1.
\]
The interior residual support locations $r_1,\ldots,r_{s-1}$ are indexed in increasing order of their squared magnitudes. These squares are distinct by the preceding subsection. Throughout the rest of this section, $i$ indexes individual coordinates, while $j$ indexes support locations. Zero entries remain part of the fixed-size tuple but contribute nothing to the identities
\[
    Q=\sum_{j=0}^s m_jy_j,\qquad \sum_{j=0}^s m_jr_j^{2\ell+3}=0 \quad(\ell=0,\ldots,k-1).
\]

In this notation, \eqref{eq:upper_stationarity} and \eqref{eq:R_form} give
\begin{equation}\label{eq:h_values}
    \begin{aligned}
        h(y_0)&=r_0-\frac{2r_0-1}{2c},\\
        h(y_j)&=r_j,
            &&j=1,\ldots,s-1,\\
        h(1)&=-1-\frac{r_0}{2cm_s},\\
        R'(r_j)&=2c\left(1-2r_jh'(y_j)\right),
            &&j=0,\ldots,s.
    \end{aligned}
\end{equation}

\begin{lemma}[Minimal support]\label{lem:upper_minimal_support}
    Under the hypotheses of Lemma~\ref{lem:upper_local_optimality}, the residual support has size $s=k$. Thus $x$ has exactly $k+1=(g-1)/2$ distinct nonzero values.
\end{lemma}
\begin{proof}
The moment constraints already force $s\geq k$, so it remains to show that a local maximizer cannot have more.

\subsubsection{At most one extra interior support location}

First suppose that $s\geq k+2$, so that $r_1,\ldots,r_{k+1}$ are distinct nonzero interior support locations. Put
\[
    V(y):=\prod_{j=1}^{k+1}(y-y_j).
\]
At each support location $r_j$, $j=1,\ldots,k+1$, choose one copy and give it initial velocity
\[
    \frac{1}{y_jV'(y_j)}.
\]
Give every other individual entry initial velocity zero, including the top and all minimum entries. Denote the resulting individual-coordinate velocity vector by $v$. This vector is tangent to the moment constraints, since for each $\ell=0,\ldots,k-1$,
\[
    \sum_{i=1}^N(2\ell+3)x_i^{2\ell+2}v_i=(2\ell+3)\sum_{j=1}^{k+1}\frac{y_j^\ell}{V'(y_j)}=0.
\]
The last equality follows from \eqref{eq:lagrange_cancellation}, applied to $f(y)=y^\ell$ at the $k+1$ nodes $y_1,\ldots,y_{k+1}$.
Since $v_1=0$ and all minimum entries have zero initial velocity, Lemma~\ref{lem:upper_local_optimality} applies to this vector and yields
\begin{equation}\label{eq:too_many_distinct_values_1}
    \sum_{j=1}^{k+1}\frac{R'(r_j)}{y_j^2V'(y_j)^2}\geq0.
\end{equation}
There are no multiplicity factors in this sum, since only one copy at each selected support location has nonzero initial velocity.

Together with $h(0)=0$, the identities in \eqref{eq:h_values} give the partial fraction decomposition
\[
    \frac{h(y)^2-y}{y^2V(y)^2}=-\frac{1}{V(0)^2y}-\frac{1}{2c}\sum_{j=1}^{k+1}\frac{R'(r_j)}{y_j^2V'(y_j)^2(y-y_j)}.
\]
Indeed, the numerator vanishes at every $y=y_j$ and at $y=0$, so the left-hand side has poles of order at most one at these points. The coefficients of the partial fraction decomposition are determined by the residues at these poles. Since $\deg h\leq k$ and $\deg V=k+1$, the left-hand side is $O(y^{-4})$ as $y\to\infty$. The vanishing of the $y^{-1}$ term on the right-hand side therefore gives
\begin{equation}\label{eq:too_many_distinct_values_2}
    \sum_{j=1}^{k+1}\frac{R'(r_j)}{y_j^2V'(y_j)^2}=-\frac{2c}{V(0)^2}<0,
\end{equation}
contradicting \eqref{eq:too_many_distinct_values_1}.

Thus there are at most $k$ distinct nonzero interior support locations. Including the minimum $r_s=-1$, the residual support therefore has size $s\in\{k,k+1\}$. It remains to exclude the latter case $s=k+1$.

\subsubsection{Excluding the case of one extra interior support location}
For this case we also need to include the top support location $r_0$ in the tangent vector. Indeed, we redefine 
$V(y):=\prod_{j=0}^{k}(y-y_j)$, and give every copy at support location $r_j$ the same initial velocity
\[
    u_j:=\frac{1}{m_j y_j V'(y_j)}, \qquad j=0,\dots,k,
\]
with all minimum and zero entries remaining fixed.

The resulting vector is tangent to the moment constraints because
\[
    \sum_{j=0}^{k} m_j r_j^{2\ell} u_j = \sum_{j=0}^{k} \frac{y_j^{\ell-1}}{V'(y_j)} = 0, \quad \ell=1,2,\dots,k
\]
by \eqref{eq:lagrange_cancellation}, applied to $f(y)=y^{\ell-1}$ at the $k+1$ nodes $y_0,\ldots,y_k$. Consequently, Lemma~\ref{lem:upper_local_optimality} implies the inequality
\[
    \sum_{j=0}^{k} m_j R'(r_j) u_j^2 -2u_0^2 \geq 0,
\]
or equivalently,
\begin{equation}\label{eq:one_extra_value_1}
    \frac{R'(r_0)-2}{y_0^2 V'(y_0)^2}+\sum_{j=1}^{k} \frac{R'(r_j)}{m_j y_j^2 V'(y_j)^2}  \geq 0.
\end{equation}

We will show that the left-hand side of \eqref{eq:one_extra_value_1} is strictly negative, contradicting
local optimality and excluding the case $s=k+1$.

To this end, we will eliminate $R'(r_0)$ and express the remaining terms as a sum over interior support locations with positive coefficients, together with a strictly negative correction. The splitting condition will then allow us to remove the multiplicities from an upper bound. We first record how the moment constraints relate the multiplicities to the support locations. The sum
\[
    \sum_{j=0}^{k+1} \frac{m_j r_j^3}{y-y_j}
\]
has vanishing coefficients for $y^{-1},y^{-2},\ldots,y^{-k}$ when expanded at infinity. Putting it over a common denominator gives
\begin{equation}\label{eq:frac_moment_identity}
    \sum_{j=0}^{k+1} \frac{m_j r_j^3}{y-y_j} = \frac{L(y)}{(y-1)V(y)}
\end{equation}
for some polynomial $L$. The left-hand side is $O(y^{-k-1})$ at infinity and the denominator has degree $k+2$, so $\deg L\leq1$. Comparing residues gives
\begin{equation}\label{eq:L_values}
    L(y_j)=m_jr_j^3(y_j-1)V'(y_j), \qquad j=0,\ldots,k.
\end{equation}

We now show that $L$ is linear and its root lies in $(0,1)$. Any consecutive pair of poles in \eqref{eq:frac_moment_identity} whose residues have the same sign forces a zero of $L$ between them. Since $L$ is nonzero and $\deg L\leq1$, there can be at most one such pair. There are $k+1$ adjacent pairs, so the residue signs have at least $k$ sign changes. These are precisely the signs of the support locations $r_j$, ordered by their squares $y_j$.

On the other hand, the values $h(y_j)$, read in the same order, have at most $k-1$ sign changes, since $\deg h\leq k$ and $h(0)=0$ by \eqref{eq:R_form}. By \eqref{eq:h_values}, their signs agree with those of $r_j$ at every interior support location and at the minimum. At the top, $Q\geq r_0^2$ gives
\[
    2c h(y_0)=2cr_0-(2r_0-1)\leq2(r_0-1)-(2r_0-1)=-1<0.
\]
Thus the two sequences differ only in their final sign, at $y_0$. Reversing that sign changes the count by exactly one, so the support signs have exactly $k$ changes and the values of $h$ have exactly $k-1$. In particular, reversing the final sign removes a sign change.

The support location immediately preceding the top in square order is therefore negative, so its square is at most $1$. Since the minimum support location $-1$ is present, the largest residual square must equal $1$. Consequently,
\[
    0<y_1<\cdots<y_k<y_{k+1}=1<y_0.
\]
Exactly one of the $k+1$ adjacent pairs has equal signs. The final pair has opposite signs, so this same-sign pair lies among the residual poles in $(0,1]$. It forces a zero of $L$ between them. Hence $L$ is linear and its unique root lies in $(0,1)$.

Now write $L(y)=\lambda(y-y^*)$ with $y^*\in(0,1)$. By \eqref{eq:L_values},
\[
    L(y_0)=r_0^3(y_0-1)V'(y_0)>0,
\]
so $\lambda>0$.

We next seek an interpolation identity that expresses the top derivative in terms of the interior derivatives. For $j=1,\ldots,k$, \eqref{eq:h_values} gives
\[
    \frac{R'(r_j)}{r_j}=2c\left(\frac{h(y_j)}{y_j}-2h'(y_j)\right).
\]
This suggests defining
\[
    K(y):=\frac{h(y)}{y}-2h'(y).
\]
Since $h(0)=0$, this is a polynomial of degree at most $k-1$. Applying \eqref{eq:lagrange_cancellation} to $K$ at the $k+1$ nodes $y_0,\ldots,y_k$ gives
\[
    \sum_{j=0}^k\frac{K(y_j)}{V'(y_j)}=0.
\]
At the top, \eqref{eq:h_values} gives
\[
    2cK(y_0)=\frac{R'(r_0)}{r_0}-\frac{2r_0-1}{r_0^2}.
\]
Thus the interpolation identity becomes
\begin{equation}\label{eq:upper_top_elimination}
    \sum_{j=0}^k\frac{R'(r_j)}{r_jV'(y_j)}=\frac{2r_0-1}{r_0^2V'(y_0)}.
\end{equation}

We use \eqref{eq:upper_top_elimination} to eliminate $R'(r_0)$ from the left-hand side of \eqref{eq:one_extra_value_1}. For each interior summand, \eqref{eq:L_values} gives the coefficient identity
\[
    1-\frac{m_jr_j^3V'(y_j)}{r_0^3V'(y_0)}=1-\frac{(y_0-1)L(y_j)}{(y_j-1)L(y_0)}=\frac{(1-y^*)(y_0-y_j)}{(y_0-y^*)(1-y_j)}.
\]
Consequently,
\begin{equation}\label{eq:upper_interior_reduction}
    \begin{aligned}
        &\frac{R'(r_0)-2}{y_0^2V'(y_0)^2}
        +\sum_{j=1}^k\frac{R'(r_j)}{m_jy_j^2V'(y_j)^2}\\
        &\qquad=\frac{1-y^*}{y_0-y^*}
        \sum_{j=1}^k\frac{(y_0-y_j)R'(r_j)}{m_j(1-y_j)y_j^2V'(y_j)^2}
        -\frac{1}{r_0^5V'(y_0)^2}.
    \end{aligned}
\end{equation}
All coefficients in the sum on the right are positive. If $m_j\geq2$, the splitting condition \eqref{eq:upper_interior_splitting} gives $R'(r_j)\geq0$, so $R'(r_j)/m_j\leq R'(r_j)$. The same inequality holds with equality when $m_j=1$. We may therefore remove the multiplicities from this sum in an upper bound. Since the final term in \eqref{eq:upper_interior_reduction} is strictly negative, it remains to prove that
\begin{equation}\label{eq:upper_unweighted_sum}
    \sum_{j=1}^k\frac{(y_0-y_j)R'(r_j)}{(1-y_j)y_j^2V'(y_j)^2}\leq0.
\end{equation}

As in the proof of \eqref{eq:too_many_distinct_values_2}, we construct a rational function whose residues are equal to the summands in \eqref{eq:upper_unweighted_sum}. Specifically, consider the function
\[
    f(y) := \frac{-2c(y-y_0)(h(y)^2-y)}{y^2(y-1)V(y)^2}.
\]
The residues of $f(y)$ at $y=y_j$ for $j=1,\dots,k$ are equal to the summands in \eqref{eq:upper_unweighted_sum}. The remaining possible poles are located at $y\in\{0,1,y_0\}$. Since $f(y)=O(y^{-4})$ as $y\to\infty$, the sum of all residues must be zero. Consequently, the sum in \eqref{eq:upper_unweighted_sum} equals the negative of the sum of the residues at these three points. We directly calculate these residues:
\begin{align*}
    \text{Res}(f,0) &= \frac{2cy_0}{V(0)^2}\geq0,\\
    \text{Res}(f,1) &= \frac{2c(y_0-1)(h(1)^2-1)}{V(1)^2},\\
    \text{Res}(f,y_0) &= -\frac{2c(h(y_0)^2-y_0)}{y_0^2(y_0-1)V'(y_0)^2}.
\end{align*}
Now we observe from \eqref{eq:h_values} that
\[
    h(1)^2-1 >\left(\frac{r_0}{2cm_s}\right)^2, \quad h(y_0)^2-y_0 < \left(\frac{2r_0-1}{2c}\right)^2.
\]
Also \eqref{eq:L_values} along with $L(y)=\lambda(y-y^*)$ implies that
\[
    \lambda(1-y^*)=-m_s V(1), \quad \lambda(y_0-y^*)=r_0^3(y_0-1) V'(y_0).
\]
Substituting these expressions into the residues calculated above, we obtain
\begin{align*}
    \text{Res}(f,1) & \geq \frac{2c(y_0-1)\left(\frac{r_0}{2cm_s}\right)^2}{V(1)^2} = \frac{y_0(y_0-1)}{2 c\lambda^2 (1-y^*)^2}, \\
    \text{Res}(f,y_0) & \geq -\frac{2c\left(\frac{2r_0-1}{2c}\right)^2}{y_0^2(y_0-1)V'(y_0)^2} = -\frac{y_0(y_0-1)(2r_0-1)^2}{2c\lambda^2 (y_0-y^*)^2}.
\end{align*}
and therefore
\[
    \text{Res}(f,1)+\text{Res}(f,y_0) \geq \frac{y_0(y_0-1)}{2c\lambda^2} \left( \frac{1}{(1-y^*)^2} - \frac{(2r_0-1)^2}{(y_0-y^*)^2} \right) \geq0,
\]
where the last inequality follows from
\[
    (y_0-y^*) - (2r_0-1)(1-y^*) = (r_0-1)(r_0-1+2y^*)>0.
\]
This finally proves \eqref{eq:upper_unweighted_sum}, which contradicts \eqref{eq:one_extra_value_1} and \eqref{eq:upper_interior_reduction}, and therefore excludes the case $s=k+1$.
\end{proof}

\subsection{A signed tail-counting function and splitting inequalities}

\begin{lemma}[Multiplicity and sign alternation]\label{lem:upper_multiplicity_floor}
    Under the hypotheses of Lemma~\ref{lem:upper_local_optimality}, every residual support location has multiplicity at least two. The distinct nonzero values, ordered by their magnitudes, have alternating signs, and the largest residual magnitude is $1$.
\end{lemma}
\begin{proof}
We construct the signed tail-counting function
\begin{equation}\label{eq:root_counting_function}
    C(t):= \sum_{|x_i| >t} \sign(x_i), t>0.
\end{equation}
This function captures the moment constraints by the following identity:
\[
    \int_0^\infty t^{2\ell}C(t)\,dt = \frac{1}{2\ell+1}\sum_{i=1}^N x_i^{2\ell+1}=0, \quad \ell=1,2,\dots,k.
\]
A first consequence of this identity is that $C(t)$ must change sign at least $k$ times for $t>0$. Otherwise, there would exist a polynomial $S$ of degree at most $k-1$ such that $t^2S(t^2)C(t)$ is nonnegative for $t>0$ and strictly positive on some interval, contradicting the integral identity.

On the other hand, by Lemma~\ref{lem:upper_minimal_support}, the function $C$ has exactly $k$ jump locations on the interval $(0,r_0)$, namely the distinct residual magnitudes $|r_j|$. Since it has at least $k$ sign changes, every jump must connect nonzero values of opposite signs. At each such location $t=|r_j|$, the jump has absolute size $m_j$, so we obtain
\[
    m_j=\left|C(|r_j|^-)-C(|r_j|^+)\right| \geq 1-(-1)=2,\qquad j=1,\ldots,k.
\]
Another immediate consequence is that the distinct nonzero values, ordered by their magnitudes, must have alternating signs. Since $C(t)=1$ between the largest residual magnitude and $r_0$, crossing the largest residual jump toward smaller $t$ changes $C$ from positive to negative. The corresponding residual support location is therefore negative. As $-1$ is present and is the minimum value, the largest residual magnitude must equal $1$. Therefore, in our indexing convention
\[
    0<y_1<y_2<\cdots<y_k=1<y_0,
\]
the signs of $r_j$ satisfy
\[
    \sign(r_j) = (-1)^{k-j+1},\qquad j=1,\ldots,k.
\]
\end{proof}

By Lemma~\ref{lem:upper_multiplicity_floor}, we can use the splitting stability inequality \eqref{eq:upper_interior_splitting} to conclude
\[
    R'(r_j)\geq 0,\qquad j=1,\ldots,k-1,
\]
and then use \eqref{eq:h_values} to obtain
\begin{equation}\label{eq:splitting_stability_2}
    2r_jh'(y_j)\leq 1,\qquad j=1,\ldots,k-1.
\end{equation}

We then proceed to convert this relation into inequalities concerning only $r_j$ and $m_j$. To this end we interpolate the polynomial $h(y)/y$ (with degree at most $k-1$) at the points $y_1,\ldots,y_{k}$ using the Lagrange interpolation formula. In fact \eqref{eq:h_values} gives
\[
    \frac{h(y_j)}{y_j}=\frac1{r_j}\quad(j=1,\ldots,k-1), \qquad \frac{h(1)}1=-1-\frac{r_0}{2cm_k},
\]
and therefore \eqref{eq:lagrange_interpolation} implies
\begin{equation}\label{eq:h_interpolation}
    \frac{h(y)}{y}=\sum_{j=1}^{k} \frac{L_j(y)}{r_j} - \frac{r_0}{2cm_k}L_k(y),
\end{equation}
where
\[
    L_j(y) = \prod_{\substack{\ell=1\\\ell\ne j}}^{k} \frac{y-y_\ell}{y_j-y_\ell},\qquad j=1,\ldots,k.
\]
We differentiate both sides of \eqref{eq:h_interpolation} with respect to $y$, and then use \eqref{eq:h_values} to obtain for $j=1,2,\ldots,k-1$
\[
    h'(y_j)=\frac1{r_j}+y_j\sum_{\ell=1}^k\frac{L_\ell'(y_j)}{r_\ell}-\frac{r_0y_j}{2cm_k}L_k'(y_j).
\]
Substituting this relation into \eqref{eq:splitting_stability_2} yields
\begin{equation}\label{eq:splitting_stability_3}
    1+2r_jy_j\sum_{\ell=1}^k\frac{L_\ell'(y_j)}{r_\ell}-\frac{r_0r_jy_j}{cm_k}L_k'(y_j)\leq 0,\qquad j=1,\ldots,k-1.
\end{equation}
Let $V(y):=\prod_{\ell=1}^{k} (y-y_\ell)$. We express the derivatives of the Lagrange basis polynomials in \eqref{eq:splitting_stability_3} using \eqref{eq:lagrange_derivatives}, and recall the definition of $c$ in Lemma~\ref{lem:upper_local_optimality}, to obtain
\begin{equation}\label{eq:splitting_stability_4}
    1+2y_j\sum_{\substack{\ell=1 \\ \ell\neq j}}^{k}\frac{1+\frac{r_j V'(y_j)}{r_\ell V'(y_\ell)}}{y_j-y_\ell}-\frac{Qy_j}{m_k(r_0-1)(y_j-1)}\frac{r_jV'(y_j)}{V'(1)}\leq 0,\, j=1,\ldots,k-1
\end{equation}
where $Q=\sum_{j=0}^{k} m_j y_j$ is the second moment.

The expression $\frac{r_j V'(y_j)}{r_\ell V'(y_\ell)}$ in \eqref{eq:splitting_stability_4} can be further simplified. Indeed, since $\deg L_j=k-1$, the moment constraints imply
\[
    0= \sum_{\ell=0}^{k}m_\ell r_\ell^3 L_j(y_\ell) =r_0^3L_j(y_0)+m_j r_j^3=\frac{r_0^3V(y_0)}{(y_0-y_j)V'(y_j)}+m_j r_j^3, \, j=1,\ldots,k.
\]
Therefore 
\[
    m_j r_j^3(y_0-y_j)V'(y_j) = -r_0^3V(y_0)
\]
is independent of $j$, leading to the identity
\begin{equation}\label{eq:vandermonde_simplification}
    \frac{r_j V'(y_j)}{r_\ell V'(y_\ell)} = \frac{m_\ell y_\ell (y_0-y_\ell)}{m_j y_j (y_0-y_j)}
\end{equation}
for all $j,\ell=1,\ldots,k$. 

We then substitute the simplification \eqref{eq:vandermonde_simplification} into \eqref{eq:splitting_stability_4} to obtain inequalities that involve only the variables $r_j$, $m_j$, and $y_j=r_j^2$, namely
\begin{equation}\label{eq:splitting_stability_5}
    1+2y_j\sum_{\substack{\ell=1 \\ \ell\neq j}}^{k}\frac{1}{y_j-y_\ell}+\frac{2}{m_j (y_0-y_j)}\sum_{\substack{\ell=1 \\ \ell\neq j}}^{k}\frac{m_\ell y_\ell (y_0-y_\ell)}{y_j-y_\ell}\leq\frac{Q(r_0+1)}{m_j (y_0-y_j)(1-y_j)}.
\end{equation}
The inequality \eqref{eq:splitting_stability_5}, which holds for all $j=1,\ldots,k-1$, is the final form of the splitting stability condition. This will allow us to bound the second moment $Q$ from below. Combined with an estimate relating $r_0$ to $Q$, obtained from the moment constraints, this will thus bound the objective $c=r_0(r_0-1)/Q$ from above.

\subsection{Polynomial comparison and block estimates}

We use the splitting inequality \eqref{eq:splitting_stability_5} to bound the number of residual squares at or below an interior cutoff.

\begin{lemma}[Polynomial comparison]\label{lem:upper_polynomial_comparison}
    Let $x$ be a local maximizer as in Lemma~\ref{lem:upper_local_optimality}, with the notation above. Fix $1\leq j<k$ and put $t:=y_j$. Suppose that $S$ is a real polynomial of degree $d\geq1$, with $S(-s)=(-1)^dS(s)$, satisfying
    \begin{equation}\label{eq:P_envelope}
        S(1)=1,\qquad |S(s)|\leq\prod_{\ell>j}\frac{y_\ell-t}{y_\ell-ts^2}, \qquad |s|\leq1.
    \end{equation}
    Then
    \begin{equation}\label{eq:Q_lower_bound_2}
        2j\leq d+1+\sqrt{\sum_{\ell>j}\frac{m_\ell y_\ell}{y_\ell-t}+\frac{Q}{(1-t)^2}}.
    \end{equation}
    All sums and products range over residual indices $1,\ldots,k$.
\end{lemma}
\begin{proof}
For $d>2j-1$, the conclusion is immediate. We therefore assume $d\leq2j-1$. We first separate the contributions of the smaller squares $y_\ell<t$ from those of the larger squares $y_\ell>t$.

The terms involving the larger squares admit simple bounds. Indeed, $m_j\ge2$, and
\[
    \frac{y_0-y_\ell}{y_0-t}\le1\quad(\ell>j),\qquad \frac{r_0+1}{y_0-t}\le\frac{2}{1-t}.
\]
The second inequality follows from $y_0=r_0^2$ and
\[
    2(y_0-t)-(r_0+1)(1-t) =(r_0-1)(2r_0+1+t)\ge0.
\]
Consequently, after moving the terms with $\ell>j$ to the right in \eqref{eq:splitting_stability_5}, we obtain
\begin{equation}\label{eq:block_splitting}
    \begin{aligned}
        &1+2t\sum_{\ell<j}\frac{1}{t-y_\ell}+\frac{2}{m_j(y_0-t)}\sum_{\ell<j}\frac{m_\ell y_\ell(y_0-y_\ell)}{t-y_\ell}\\
        &\qquad\leq \sum_{\ell>j}\frac{2t+m_\ell y_\ell}{y_\ell-t}+\frac{Q}{(1-t)^2}.
    \end{aligned}
\end{equation}
All sums here range over residual indices $1,\ldots,k$.

The left-hand side of \eqref{eq:block_splitting} depends on the unknown spacing of the first $j$ squares. To obtain a bound in terms of the block size $j$, we will compare it with the endpoint derivative of a polynomial of degree at most $j-1$, using Lagrange interpolation at these squares.

Define $V_j(y):=\prod_{\ell\leq j}(y-y_\ell)$. For any polynomial $q$ of degree at most $j-1$ satisfying $q(t)=1$, \eqref{eq:lagrange_interpolation} and \eqref{eq:lagrange_derivatives} give
\[
    q'(t) = \sum_{\ell<j} \left(\frac{1}{t-y_\ell}+\frac{q(y_\ell)}{t-y_\ell} \frac{V_j'(t)}{V_j'(y_\ell)} \right).
\]
To compare this expression with the left-hand side of \eqref{eq:block_splitting}, it suffices to require
\begin{equation}\label{eq:q_condition}
    \left|q(y_\ell)\frac{V_j'(t)}{V_j'(y_\ell)}\right| \leq \frac{m_\ell y_\ell(y_0-y_\ell)}{m_jt(y_0-t)},\quad \ell=1,2,\dots,j-1.
\end{equation}
Using \eqref{eq:vandermonde_simplification} and the fact that $V'(y_\ell)=V_j'(y_\ell)\prod_{\mu>j}(y_\ell-y_\mu)$ holds for $\ell=1,2,\dots,j$, we can rewrite \eqref{eq:q_condition} as
\begin{equation}\label{eq:q_condition_2}
    \left|q(y_\ell)\right| \leq \sqrt{\frac{t}{y_\ell}}\prod_{\mu>j}\frac{y_\mu-t}{y_\mu-y_\ell},\quad \ell=1,2,\dots,j-1.
\end{equation}

Hence for every polynomial $q$ of degree at most $j-1$ satisfying $q(t)=1$ and \eqref{eq:q_condition_2}, we have
\begin{equation}\label{eq:block_splitting_2}
    1+2tq'(t)\leq \sum_{\ell>j}\frac{2t+m_\ell y_\ell}{y_\ell-t}+\frac{Q}{(1-t)^2}.
\end{equation}

We now construct $q$ with a large derivative at $t$. We use the polynomial factor $S$ to enforce the pointwise bounds \eqref{eq:q_condition_2}, and a Chebyshev polynomial to obtain a large endpoint derivative. We define $q$ by
\begin{equation}\label{eq:q_construction}
    s q(ts^2)=S(s)T_{2j-1-d}(s),
\end{equation}
By the degree and parity statements in Lemma~\ref{lem:chebyshev_polynomials}, the product $S(s)T_{2j-1-d}(s)$ is odd and has degree $2j-1$. Dividing it by $s$ therefore gives an even polynomial of degree $2j-2$, so \eqref{eq:q_construction} defines a polynomial $q$ of degree at most $j-1$.

Equations \eqref{eq:chebyshev_basic_properties} and \eqref{eq:P_envelope} give $q(t)=1$ and, on evaluating \eqref{eq:q_construction} at $s=\sqrt{y_\ell/t}$, the bound \eqref{eq:q_condition_2}. Thus $q$ satisfies all the hypotheses of \eqref{eq:block_splitting_2}. Finally, differentiating \eqref{eq:q_construction} at $s=1$ and applying the endpoint identities in \eqref{eq:chebyshev_basic_properties}, we obtain
\begin{equation}\label{eq:q_endpoint_derivative}
    1+2tq'(t)=S'(1)+(2j-1-d)^2.
\end{equation}
The envelope condition \eqref{eq:P_envelope} gives a lower bound on $S'(1)$. In fact, we have
\begin{equation}\label{eq:P_derivative_lower_bound}
    S'(1)\geq \left.\frac{d}{ds}\prod_{\ell>j}\frac{y_\ell-t}{y_\ell-ts^2}\right|_{s=1} = \sum_{\ell>j} \frac{2t}{y_\ell - t}.
\end{equation}

Therefore, combining \eqref{eq:block_splitting_2}, \eqref{eq:q_endpoint_derivative} and \eqref{eq:P_derivative_lower_bound}, we obtain
\begin{equation}\label{eq:Q_lower_bound_1}
    \sum_{\ell>j}\frac{m_\ell y_\ell}{y_\ell-t}+\frac{Q}{(1-t)^2} \geq (2j-1-d)^2.
\end{equation}
Taking square roots and using $d\leq2j-1$ gives \eqref{eq:Q_lower_bound_2}.
\end{proof}

We will now proceed to construct an explicit polynomial $S$ that satisfies the envelope condition \eqref{eq:P_envelope}, and estimate its degree. To this end, for all integers $\delta\geq3$, we define
\begin{equation}\label{eq:P_construction}
    S_\delta(s) = \frac{-T_\delta(s\cos(\pi/\delta))}{1+\lambda_\delta(1-s^2)},
\end{equation}
where 
\[
    \lambda_\delta = \frac{\cos^2(\pi/\delta)}{\cos^2(\pi/(2\delta))-\cos^2(\pi/\delta)}.
\]

The parameters $\lambda_\delta$ increase to infinity and satisfy a useful uniform bound. Indeed, putting $\theta=\pi/(2\delta)\in(0,\pi/6]$, we have
\begin{equation}\label{eq:lambda_delta_bound}
    \frac{1}{1+\lambda_\delta}
    =4\sin^2\theta-\tan^2\theta
    \leq3\sin^2\theta<3\theta^2
    =\frac{3\pi^2}{4\delta^2}.
\end{equation}
The function $4\sin^2\theta-\tan^2\theta$ has derivative
\[
    2\tan\theta\left(4\cos^2\theta-\sec^2\theta\right)>0
    \qquad(0<\theta\leq\pi/6).
\]
Thus $\lambda_\delta$ is strictly increasing in $\delta$, and \eqref{eq:lambda_delta_bound} implies that $\lambda_\delta\to\infty$.

$S_\delta(s)$ is actually a polynomial of degree $\delta-2$ because the roots of the denominator are $s=\pm \cos(\pi/(2\delta))/\cos(\pi/\delta)$, which are also roots of the numerator $T_\delta(s\cos(\pi/\delta))$. $S_\delta(1)=1$ follows from \eqref{eq:chebyshev_cosine}.

Now for every $\ell>j$, we define $\delta_{j,\ell}:=\min\{\delta\geq 3 : \lambda_\delta \geq \frac{t}{y_\ell-t}\}$. This is well-defined since $\lambda_\delta$ is an increasing function of $\delta$ and tends to infinity as $\delta\to\infty$. Therefore
\[
    |S_{\delta_{j,\ell}}(s)|\leq \frac{1}{1+\lambda_{\delta_{j,\ell}}(1-s^2)}\leq \frac{y_\ell-t}{y_\ell-ts^2}, \quad s\in[-1,1].
\]
We can now define the polynomial
\[
    S(s) = \prod_{\ell>j} S_{\delta_{j,\ell}}(s),
\]
which satisfies the envelope condition \eqref{eq:P_envelope}, has degree $\sum_{\ell>j} (\delta_{j,\ell}-2)=\sum_{\ell>j}\delta_{j,\ell} -2(k-j)$, and has parity matching its degree. Applying Lemma~\ref{lem:upper_polynomial_comparison} gives a concrete inequality, namely
\begin{equation}\label{eq:Q_lower_bound_3}
    2k-1 \leq \sum_{\ell>j} \delta_{j,\ell} + \sqrt{\sum_{\ell>j}\frac{m_\ell y_\ell}{y_\ell-t}+\frac{Q}{(1-t)^2}}
\end{equation}
for all $j=1,2,\dots,k-1$.

\subsection{Averaging the block estimates}
We cannot directly use \eqref{eq:Q_lower_bound_3} to obtain a simple bound on $Q$ because the quantity $\frac{m_\ell y_\ell}{y_\ell-t}$ cannot be easily controlled. Instead, we take a linear combination of the inequalities for different values of $j$ to average out these terms. 

Let $j^*:=\min\{1\leq j\leq k : y_j \geq 1-1/k\}$. For $j=1,2,\dots,j^*-1$, we multiply the inequality \eqref{eq:Q_lower_bound_3} by the weight $w_j:=y_{j+1}-y_j$ and sum over $j$ to obtain a combined inequality, namely
\begin{equation}\label{eq:Q_lower_bound_4}
    (2k-1)(y_{j^*}-y_1) \leq \sum_{j=1}^{j^*-1} w_j \sum_{\ell>j} \delta_{j,\ell} + \sum_{j=1}^{j^*-1} w_j \sqrt{\sum_{\ell>j}\frac{m_\ell y_\ell}{y_\ell-y_j}+\frac{Q}{(1-y_j)^2}}.
\end{equation}
We first deal with the first term on the right-hand side. By the definition of $\delta_{j,\ell}$ and the monotonicity of $\lambda_\delta$, we have
\[
    \delta_{j,\ell} =3+\sum_{\delta\geq3} \mathbf{1}_{\{\lambda_\delta < \frac{y_j}{y_\ell-y_j}\}}.
\]
Therefore, we can write the sum as
\begin{align*}
    &\phantom{=} \sum_{j=1}^{j^*-1} w_j \sum_{\ell>j} \delta_{j,\ell} \\
    &= 3\sum_{j=1}^{j^*-1} (k-j)w_j +\sum_{j=1}^{j^*-1} (y_{j+1}-y_j) \sum_{\ell>j}\sum_{\delta\geq3} \mathbf{1}_{\{\lambda_\delta < \frac{y_j}{y_\ell-y_j}\}}\\
    &\leq 3\sum_{j=1}^{k-1} (k-j)(y_{j+1}-y_j)+\sum_{\delta\geq3}\sum_{\ell=2}^{k} \sum_{j=1}^{\min(j^*,\ell)-1} (y_{j+1}-y_j)  \mathbf{1}_{\{y_\ell-y_j < y_\ell/(1+\lambda_\delta)\}}\\
    &\leq 3\sum_{j=1}^k y_j -3k y_1+\sum_{\delta\geq3}\sum_{\ell=2}^{k} \frac{y_\ell}{1+\lambda_\delta}\\
    &\leq -3ky_1+\left(\sum_{j=1}^k y_j\right)\left(3+\sum_{\delta\geq3} \frac{1}{1+\lambda_\delta}\right).
\end{align*}
The second inequality holds because for every fixed $\ell$ and $\delta$, all intervals $(y_j, y_{j+1})$ contributing to the sum lie in the interval $(y_\ell - y_\ell/(1+\lambda_\delta), y_\ell)$ and they are disjoint, so the sum of their lengths is at most $y_\ell/(1+\lambda_\delta)$.

Combining this inequality with \eqref{eq:Q_lower_bound_4}, $Q\geq y_0+2\sum_{j=1}^k y_j$ and $y_{j^*} \geq 1-1/k$ yields
\begin{equation}\label{eq:Q_lower_bound_5}
    \frac{(2k-1)(k-1)}{k}+(k+1)y_1 \leq \frac{C_*}{2}\left(Q-y_0\right) + \sum_{j=1}^{j^*-1} w_j \sqrt{\sum_{\ell>j}\frac{m_\ell y_\ell}{y_\ell-y_j}+\frac{Q}{(1-y_j)^2}},
\end{equation}
where we write $C_* := 3+\sum_{\delta\geq3} \frac{1}{1+\lambda_\delta}$ for simplicity.

We now proceed to bound the second term on the right-hand side of \eqref{eq:Q_lower_bound_5}. We assume $j^*\geq2$; otherwise the sum is empty and there is nothing to bound. Using $\sqrt{a+b}\leq\sqrt a+\sqrt b$, we obtain
\begin{equation}\label{eq:Q_lower_bound_5_sqrt}
    \sum_{j=1}^{j^*-1} w_j \sqrt{\sum_{\ell>j}\frac{m_\ell y_\ell}{y_\ell-y_j}+\frac{Q}{(1-y_j)^2}} \leq \sum_{j=1}^{j^*-1} w_j \sqrt{\sum_{\ell>j}\frac{m_\ell y_\ell}{y_\ell-y_j}} + \sum_{j=1}^{j^*-1} \frac{w_j \sqrt{Q}}{1-y_j}.
\end{equation}

Separating the last gap and comparing the remaining terms with an integral gives
\[
    \sum_{j=1}^{j^*-1} \frac{w_j \sqrt{Q}}{1-y_j} \leq \sqrt{Q}\left(\frac{y_{j^*}-y_{j^*-1}}{1-y_{j^*-1}} + \int_{y_1}^{y_{j^*-1}} \frac{dy}{1-y} \right) < \sqrt{Q} (1+\log k),
\]
where we used $y_{j^*} \leq 1$ and $y_{j^*-1}< 1-1/k$.

We use H\"older's inequality to bound the first term on the right-hand side of \eqref{eq:Q_lower_bound_5_sqrt}. For simplicity of notation, we write
\[
    H_j := \sum_{\ell>j}\frac{m_\ell y_\ell}{y_\ell-y_j}.
\]
Then, using $\sum_j w_j <1$, we obtain
\begin{align*}
    &\left(\sum_{j=1}^{j^*-1} w_j \sqrt{H_j}\right)^{3/2} \leq \sum_{j=1}^{j^*-1} w_j H_j^{3/4} 
    \leq \sum_{j=1}^{j^*-1} \sum_{\ell>j} w_j \left(\frac{m_\ell y_\ell}{y_\ell-y_j}\right)^{3/4}\\
    &= \sum_{\ell=2}^{k} (m_\ell y_\ell)^{3/4} \sum_{j=1}^{\min(j^*,\ell)-1} \frac{y_{j+1}-y_j}{(y_\ell-y_j)^{3/4}}\\
    &\leq \sum_{\ell=2}^{k} (m_\ell y_\ell)^{3/4} \sum_{j=1}^{\min(j^*,\ell)-1} \int_{y_j}^{y_{j+1}} \frac{dy}{(y_\ell-y)^{3/4}}\\
    &\leq \sum_{\ell=2}^{k} (m_\ell y_\ell)^{3/4} \int_{0}^{y_\ell} \frac{dy}{(y_\ell-y)^{3/4}}
    = 4 \sum_{\ell=2}^{k} m_\ell^{3/4} y_\ell < 4 \sum_{\ell=1}^{k} m_\ell y_\ell<4Q.
\end{align*}

Combining the two estimates with \eqref{eq:Q_lower_bound_5_sqrt} and \eqref{eq:Q_lower_bound_5}, and making some trivial simplifications, we obtain
\begin{equation}\label{eq:Q_lower_bound_6}
    2k-3 < \frac{C_*}{2}\left(Q-y_0\right) + (4Q)^{2/3} + \sqrt{Q} (1+\log k).
\end{equation}

In asymptotic terms, the inequality \eqref{eq:Q_lower_bound_6} simplifies to
\begin{equation}\label{eq:Q_lower_bound_final}
    Q \geq \left(\frac{4}{C_*}-o(1)\right) k.
\end{equation}
This holds because if $Q<2k$ then the error term is $o(k)$, while if $Q\geq 2k$ then the bound holds automatically.

\subsection{Completion of the upper bound}

\begin{proof}[Proof of Theorem~\ref{thm:moment_upper_bound}]
By the initial reductions, it suffices to bound $\Phi$ uniformly in $N$ for the maximizing multisets considered above.

Having established the lower bound on $Q$ in \eqref{eq:Q_lower_bound_final}, we now proceed to control the top entry $r_0$ to finally bound the objective $c=r_0(r_0-1)/Q$. To this end, we recall the moment constraints imply that for any polynomial $q\in\operatorname{span}\{x^3,x^5,\dots,x^{2k+1}\}$, we have
\[
    \sum_{j=0}^{k} m_j q(r_j)=0.
\]
We take $q(x)=x^2 T_{2k-1}(x)$, and use \eqref{eq:chebyshev_basic_properties} to obtain
\begin{equation}
    r_0^2 T_{2k-1}(r_0) = -\sum_{j=1}^{k} m_j r_j^2 T_{2k-1}(r_j) \leq \sum_{j=1}^{k} m_j r_j^2 < Q.
\end{equation}
We now write $W:=Q/r_0^2$ and $u:=\operatorname{arcosh} r_0$, to obtain
\begin{equation}
    \exp((2k-1) u)<2\cosh ((2k-1) u) \leq 2W,
\end{equation}
and therefore
\begin{equation}
    u < \frac{\log(2W)}{2k-1}.
\end{equation}
Consequently, the objective $c=r_0(r_0-1)/Q$ can be bounded by
\begin{equation}\label{eq:c_upper_bound_1}
    c = \frac{1-\operatorname{sech} u}{W}\leq \frac{u^2}{2W} \leq \frac{(\log(2W))^2}{2(2k-1)^2 W}.
\end{equation}
It remains to bound $W$ in terms of $Q$ and invoke the lower bound \eqref{eq:Q_lower_bound_final}. In fact, if $W<2k$ then
\[
    r_0 <\cosh\left(\frac{\log(4k)}{2k-1}\right)=1+o(1),
\]
so $W = Q/r_0^2 > \left(\frac{4}{C_*}-o(1)\right) k$. If $W\geq2k$ then the same bound holds trivially because $C_*>3$.

Since $(\log(2W))^2/W$ is decreasing for sufficiently large $W$ (specifically $W>e^2/2$), we substitute the lower bound on $W$ into \eqref{eq:c_upper_bound_1}, and recall the initial definition $k=(g-3)/2$ to obtain
\begin{equation}\label{eq:c_upper_bound_final}
    c < \left(\frac{C_*}{32}+o(1)\right)\frac{(\log k)^2}{k^3}=\left(\frac{C_*}{4}+o(1)\right)\frac{(\log g)^2}{g^3}.
\end{equation}
The constant is therefore $C_*/4$. By \eqref{eq:lambda_delta_bound}, we can bound $C_*$ by
\begin{align*}
    C_* & < 3 + \sum_{\delta=3}^{5} 1/(1+\lambda_\delta) +\sum_{\delta\geq6} \frac{3\pi^2}{4\delta^2} \\
    & = 3+\frac23+(\sqrt2-1)+\frac{5-\sqrt5}{10} +\frac{3\pi^2}{4}\left(\frac{\pi^2}{6}-1-\frac14-\frac19-\frac{1}{16}-\frac{1}{25}\right)\\&\approx5.69946<\frac{40}{7}.
\end{align*}
\end{proof}

\section{Lower Bounds for the Abstract Moment Problem}
\label{sec:lower_bounds}
In this section, we establish the asymptotic lower bound for the abstract moment problem by constructing finite real lists with positive integer multiplicities.

\begin{theorem}\label{thm:moment_lower_bound}
For every sufficiently large odd integer $g$, there exist distinct nonzero support locations $r_0,r_1,\ldots,r_s\in\R$, indexed so that $r_0$ is the largest, and corresponding multiplicities $m_0,m_1,\ldots,m_s\in\Z^+$ that satisfy the moment conditions
\[
    \sum_{j=0}^s m_j r_j^k = 0, \quad k=1,3,\dots,g-2,
\]
and 
\[
    \frac{r_0(r_0+r_{\min})}{\sum_{j=0}^s m_j r_j^2} \geq (1-o(1))\frac{(\log g)^2}{g^3}.
\]
Here $r_{\min}:=\min_{0\leq j\leq s}r_j$. The corresponding list contains $m_j$ copies of each $r_j$, so its length is $N=\sum_{j=0}^s m_j$, while $s$ counts the residual support locations. Consequently, $\gamma'_g\geq(1-o(1))(\log g)^2/g^3$ as $g\to\infty$ through odd integers.
\end{theorem}

The main difficulty is to satisfy the moment conditions while keeping the multiplicities integral and the second moment small. We obtain the list from the jumps of an integer-valued step function: the absolute jump sizes give the multiplicities, while the jump directions determine the signs of the corresponding entries. Integration by parts then converts the required odd moment conditions into integral identities for the step function.

To obtain these identities, we minimize a weighted integral of the absolute value of a monic polynomial, together with a linear term that accounts for the desired largest entry. The first variation supplies the integral identities. Using the Legendre polynomial identities from Section~\ref{sec:legendre_preliminaries}, we first control this linear term; we then carry out the construction and estimate its second moment.

\subsection{Auxiliary parameters}
We need the linear term associated with the largest entry to be smaller in absolute value than the weighted integral on $[0,1]$. Increasing the weight near $1$ will provide this control; restricting the increase to a short interval will allow us to keep the second moment small.

Put $n=(g-1)/2$, and for $\tau>0$ set $a=\cosh(\tau/(2n))$, the desired largest entry. For $b\in(0,1)$ and an integer $H\geq2$, define the weight function
\[
    B(t):=\begin{cases}
    H, & t\in[b,1], \\
    1, & \text{otherwise}.
    \end{cases}
\]
For a polynomial $p(x)$, define the norm and the functional
\[
    \|p\|_B := \int_0^1 B(t) |p(t^2)| \, dt, \qquad \ell_a(p) := \int^a_1 p(t^2) \, dt,
\]
and the operator norm
\[
    \lambda_{a,n}(B) := \sup_{\deg p\leq n, p\neq 0} \frac{|\ell_a(p)|}{\|p\|_B}.
\]
We will choose the parameters so that $\lambda_{a,n}(B)<1$. This condition ensures that the linear term is dominated by the norm and makes the minimization problem in the next subsection coercive.

To estimate this operator norm, we use an even polynomial that is small on $[0,b]$, bounded by $1$ on $[b,1]$, and at least $1$ on $[1,a]$. Its smallness controls the part where $B=1$, while the larger weight supplies a factor $H^{-1}$ on $[b,1]$. For an integer $d\geq0$, define
\[
    h(t):=\frac{T_d(t/b)^2}{T_d(1/b)^2}.
\]
By the parity and monotonicity in Lemma~\ref{lem:chebyshev_polynomials} and the bound \eqref{eq:chebyshev_basic_properties}, the polynomial $h$ is even and nonnegative and satisfies
\begin{equation}\label{eq:chebyshev_bounds}
    \begin{aligned}
    0 \leq h(t) \leq 1, &\quad |t|\leq 1, \\
    h(t) \leq \frac{1}{T_d(1/b)^2}, &\quad |t|\leq b, \\
    h(t) \geq 1, &\quad |t|\geq 1.
    \end{aligned}
\end{equation}
The values of $\tau$, $b$, $H$, and $d$ will be specified in the final subsection.

For $\deg p\leq n$, the polynomial $p(t^2)h(t)$ has degree at most $2(n+d)$. Combining \eqref{eq:legendre_kernel} and \eqref{eq:legendre_reproducing}, we therefore have for all $t\in[1,a]$,
\[
    p(t^2)h(t) = \int_0^1 \left(K_{2(n+d)}(t,y)+K_{2(n+d)}(t,-y)\right) p(y^2) h(y) \, dy,
\]
and consequently,
\begin{align*}
    \ell_a(p) &= \int_1^a \frac{p(t^2)h(t)}{h(t)} \, dt \\
    &= \int^1_0 \left(\int^a_1 \frac{K_{2(n+d)}(t,y)+K_{2(n+d)}(t,-y)}{h(t)} \, dt\right) p(y^2) h(y) \, dy \\
    &= \int^1_0 D_{a,n+d}(y)p(y^2) \,dy,
\end{align*}
where
\[
    D_{a,n+d}(y) := h(y)\int^a_1 \frac{K_{2(n+d)}(t,y)+K_{2(n+d)}(t,-y)}{h(t)} \, dt.
\]
We can bound the kernel $D_{a,n+d}(y)$ for $y\in[0,1]$ by utilizing \eqref{eq:legendre_bound_1}, \eqref{eq:legendre_bound_2}, \eqref{eq:legendre_parity}, \eqref{eq:legendre_integral} and \eqref{eq:chebyshev_bounds}. Indeed, we have
\begin{align*}
    |D_{a,n+d}(y)| &\leq h(y) \sum_{j=0}^{n+d} |P_{2j}(y)| \int^a_1 \frac{(4j+1) P_{2j}(t)}{h(t)}  \, dt \\
    &\leq h(y) \sum_{j=0}^{n+d} \int^a_1 (4j+1) P_{2j}(t) \, dt \\
    &= h(y) \sum_{j=0}^{n+d} \left(P_{2j+1}(a)-P_{2j-1}(a)\right) \\
    &= h(y) \left(P_{2(n+d)+1}(a)-1\right).
\end{align*}
Along with the bound \eqref{eq:chebyshev_bounds} for $h(y)$, we obtain
\begin{equation}\label{eq:lambda_upper_bound}
    \lambda_{a,n}(B) \leq \left(P_{2(n+d)+1}(a)-1\right) \max\left(\frac{1}{H}, \frac{1}{T_d(1/b)^2}\right).
\end{equation}

\subsection{A minimization problem}
Assume that $\lambda_{a,n}(B)<1$. Among all monic real polynomials of degree $n$, consider the minimization problem
\[
    \min_p J(p), \qquad J(p):=\|p\|_B+\ell_a(p).
\]
By the definition of $\lambda_{a,n}(B)$,
\begin{equation}\label{eq:J_bounds}
    \left(1-\lambda_{a,n}(B)\right)\|p\|_B \leq J(p)\leq \left(1+\lambda_{a,n}(B)\right)\|p\|_B.
\end{equation}
Since all norms on a finite-dimensional vector space are equivalent, $J$ is coercive on the affine space of monic degree-$n$ polynomials. Thus, by continuity, a minimizer exists.

Fix a minimizer $p$. For every real polynomial $q$ of degree less than $n$, the polynomial $p+\varepsilon q$ remains monic of degree $n$, and hence
\begin{equation}\label{eq:minimizer_first_variation}
    0=\left.\frac{d}{d\varepsilon}J(p+\varepsilon q)\right|_{\varepsilon=0}
    =\int_0^1 B(t)\sign(p(t^2))q(t^2)\,dt+\ell_a(q).
\end{equation}
Differentiation under the integral follows by dominated convergence: $p(t^2)$ has only finitely many zeros, and the absolute-value difference quotient is bounded by $|q(t^2)|$.

We claim that all $n$ zeros of $p$ lie in $(0,1)$ and are simple. Indeed, let $y_1<\cdots<y_k$ be the distinct zeros of $p$ in $(0,1)$ having odd multiplicity. If $k<n$, then we can choose $\sigma\in\{-1,1\}$ so that the nonzero polynomial
\[
    q(x):=\sigma\prod_{i=1}^k(x-y_i)
\]
has the same sign as $p$ almost everywhere on $(0,1)$. Equation \eqref{eq:minimizer_first_variation} then gives
\[
    0=J(q) \geq \left(1-\lambda_{a,n}(B)\right)\|q\|_B>0,
\]
a contradiction. Therefore, all $n$ zeros of $p$ must lie in $(0,1)$ and be simple.

We then define the integer-valued step function
\[
    C(t):=\begin{cases}
        B(t)\sign(p(t^2)), & 0<t<1,\\
        1, & 1<t<a,\\
        0, & t>a.
    \end{cases}
\]
\eqref{eq:minimizer_first_variation} implies that
\begin{equation}\label{eq:step_function_moments}
    \int_0^\infty C(t)q(t^2)\,dt=0
\end{equation}
for all real polynomials $q$ of degree less than $n$.

Now the distributional derivative of the step function $C(t)$ on $(0,+\infty)$ is a linear combination of Dirac delta functions at the jumps. We write 
\[
    C'(t) = \sum_{j=0}^s \Delta_j \delta(t-t_j),
\]
where $t_0=a$ and $t_1,\ldots,t_s$ are the distinct remaining jump locations, and $\Delta_j$ is the signed jump size at $t_j$. These locations consist of the positive roots of the even polynomial $t\mapsto p(t^2)$ and the points $b,1,a$, with coincident locations counted only once. We define the support locations and their multiplicities by
\[
    r_j:=-\sign(\Delta_j)t_j,\qquad m_j:=|\Delta_j|,\qquad j=0,\ldots,s.
\]
They are nonzero and distinct because the $t_j$ are positive and distinct. The jump at $a$ has size $-1$, so $r_0=a$ and $m_0=1$.

For $\ell=0,\ldots,n-1$, taking $q(x)=x^\ell$ in \eqref{eq:step_function_moments} and integrating by parts gives
\[
    \sum_{j=0}^s m_j r_j^{2\ell+1} =-\sum_{j=0}^s \Delta_j t_j^{2\ell+1} =(2\ell+1)\int_0^\infty C(t)t^{2\ell}\,dt =0.
\]

Since all other jump locations lie in $(0,1]$ and the first moment vanishes, we have $r_0=a$ and $-1\leq r_{\min}<0$. Therefore the numerator of the expression
\[
\frac{r_0(r_0+r_{\min})}{\sum_{j=0}^s m_j r_j^2}
\]
can be bounded below by $a(a-1)$. The next two subsections will be dedicated to bounding the denominator $\sum_{j=0}^s m_j r_j^2$ from above.

\subsection{Bounding the second moment}
Write
\[
    Q:=\sum_{j=0}^s m_j r_j^2,
\]
and let $0<u_1<\cdots<u_n<1$ be the positive roots of the even polynomial $t\mapsto p(t^2)$. Thus $u_1^2,\ldots,u_n^2$ are the roots of $p$. We estimate $Q$ by considering the absolute sizes of the jumps of $C$.

The jumps at $1$ and $a$ have absolute sizes $H-1$ and $1$, respectively, and therefore contribute $(H-1)+a^2$ to $Q$. At each positive root $u_j\neq b$, the jump has absolute size $2B(u_j)$. At $b$, the jump has absolute size $H-1$ if $p(b^2)\neq0$, and $H+1=(H-1)+2$ if $p(b^2)=0$. Consequently, in both cases,
\begin{equation}\label{eq:Q_formula}
    Q=a^2+(H-1)(1+b^2)+2\sum_{j=1}^n u_j^2+2(H-1)\sum_{u_j>b}u_j^2.
\end{equation}

Let $N_b:=|\{j:u_j\geq b\}|$ be the number of these positive roots in $[b,1)$. Since $b<1$ and $u_j<1$
for every $j$, we obtain
\begin{equation}\label{eq:Q_bounds}
    2\sum_{j=1}^n u_j^2 \leq Q \leq a^2+2\sum_{j=1}^n u_j^2+2(H-1)(N_b+1).
\end{equation}
It remains to control the sum $\sum_{j=1}^n u_j^2$ of the squares of these positive roots and their number $N_b$ in the interval $[b,1)$. We will bound the sum of squares in this subsection and $N_b$ in the next subsection.

We start from a general bound on the root sum of a polynomial with roots in $(0,1)$ in terms of its sup-norm on the interval $[0,1]$. 

\begin{lemma}
    Let $q(x):=(x-q_1)\cdots(x-q_n)$ be a polynomial with roots $0<q_1<\cdots<q_n<1$. Then we have the inequality
\[
    \left|2\sum_{j=1}^n q_j-n\right| \leq \log M
\]
where $M:=4^n \sup_{x\in[0,1]} |q(x)|$.
\end{lemma}
\begin{proof}
    Consider the monic polynomial 
    \[
    \tilde{q}(z):=(4z)^n q((z+z^{-1}+2)/4)
    \]
    with all the roots on the unit circle. The supremum of $|\tilde{q}(z)|$ on the unit circle is exactly $M$.

    Now we consider the Fourier expansion of the nonnegative function $V(\theta):= \log M-\log |\tilde{q}(e^{i\theta})|$ on the unit circle. We write 
    \[
    V(\theta) = \sum_{k\in\Z} c_k e^{ik\theta},
    \]
    and calculate that $c_0 = \log M$ and $c_1 = 2\sum_{j=1}^n q_j - n$. These coefficients follow from the standard Fourier expansion \[\log(2|\sin(\theta/2)|)=-\sum_{m=1}^{\infty}\frac{\cos(m\theta)}{m},\] applied to the conjugate root pairs $e^{\pm i\phi_j}$ of $\tilde q$, where $\cos\phi_j=2q_j-1$. The nonnegativity of the function then gives $|c_1| \leq c_0$, which implies the desired inequality.
\end{proof}

We now apply this lemma to the polynomial $p(x)$ with roots $u_1^2,\ldots,u_n^2$. The sup-norm of $p$ can be bounded by comparing $p$ to a reference Chebyshev polynomial on the interval $[0,1]$. Indeed, let $p^*(x):=2^{1-2n}T_n(2x-1)$. By Lemma~\ref{lem:chebyshev_polynomials}, this polynomial is monic of degree $n$ and has sup-norm $2^{1-2n}$ on $[0,1]$. Since $p$ minimizes the functional $J$, we have $J(p)\leq J(p^*)$, and consequently \eqref{eq:J_bounds} implies
\[
    \|p\|_{B}  \leq \frac{1+\lambda_{a,n}(B)}{1-\lambda_{a,n}(B)}\|p^*\|_{B}.
\]
To convert this norm bound into sup-norm bounds, we first notice that
\[
    \|p^*\|_{B}\leq 2^{1-2n}\int^1_0 B(x)\,dx = 2^{1-2n} (b+H(1-b)).
\]
Also, we reproduce the even polynomial $p(t^2)$ using \eqref{eq:legendre_reproducing}:
\begin{align*}
    |p(t^2)| &= \left|\int_0^1 p(x^2) \big(K_{2n}(t,x) + K_{2n}(t,-x)\big) \, dx\right| \\
    &\leq \left(\int_0^1 |p(x^2)|\, dx\right) \sup_{0\leq x\leq 1}\big|K_{2n}(t,x) + K_{2n}(t,-x)\big|\\
    &\leq \|p\|_{B} \sup_{0\leq x\leq 1}\sum_{k=0}^{n}(4k+1)|P_{2k}(t)P_{2k}(x)|\\
    &\leq \|p\|_{B} \sup_{0\leq x\leq 1}\sum_{k=0}^{n}(4k+1)\\
    &=(n+1)(2n+1) \|p\|_{B}, \qquad t\in[0,1].
\end{align*}
Combining this with the bound on $\|p\|_{B}$ from above, we obtain the desired sup-norm estimate for $p$.
\[
    \sup_{t\in[0,1]} |p(t)| \leq (n+1)(2n+1) \frac{1+\lambda_{a,n}(B)}{1-\lambda_{a,n}(B)} 2^{1-2n} (b+H(1-b)).
\]

Applying the root-sum lemma now gives
\begin{equation}\label{eq:root_sum_bound}
\sum_{j=1}^n u_j^2\in[(n-\log M)/2, (n+\log M)/2],
\end{equation}
where
\begin{equation}\label{eq:M_definition}
    M := 2(n+1)(2n+1) \frac{1+\lambda_{a,n}(B)}{1-\lambda_{a,n}(B)} (b+H(1-b)).
\end{equation}

\subsection{Counting roots in small intervals}
It remains to estimate $N_b$, the number of positive roots of $t\mapsto p(t^2)$ in $[b,1)$, or equivalently the number of roots of $p$ in $[b^2,1)$. To this end, we derive a ``mass balance'' identity for the polynomial $p$ divided by one of its factors, and show that too many roots in a small interval would violate this balance.

For any positive root $u$ of $t\mapsto p(t^2)$, we choose $q(x) := \frac{p(x)}{x-u^2}$ in \eqref{eq:minimizer_first_variation} to get
\begin{align*}
    0 &=\int_0^a B(t)\sign(p(t^2))q(t^2)\,dt \\
    &= \int_0^a B(t)\sign(t-u)|q(t^2)|\,dt.
\end{align*}
Equivalently, we have
\begin{equation}\label{eq:mass_balance}
    \int_0^{u} B(t)|q(t^2)|\,dt = \int_{u}^{a} B(t)|q(t^2)|\,dt.
\end{equation}

Now assume $N_b\geq2$, and take $u^2$ to be the smallest root of $p(x)$ in the interval $[b^2,1)$. Then $u\in[b,1)$ and the mass balance identity \eqref{eq:mass_balance} implies that
\begin{align*}
    \int_0^{a} |q(t^2)|\,dt &\leq \int_0^{a} B(t)|q(t^2)|\,dt \\
    &= 2 \int_{u}^{a} B(t)|q(t^2)|\,dt \\
    &\leq 2H(a-b) \sup_{t\in[b,a]} |q(t^2)|.
\end{align*}
On the other hand, reproducing the even polynomial $t\mapsto q(a^2t^2)$ on the interval $[-1,1]$ gives the bound
\[
    \sup_{t\in[-a,a]} |q(t^2)| \leq \sum_{j=0}^{n-1} (4j+1) \int_0^1 |q(a^2t^2)|\,dt = \frac{n(2n-1)}{a} \int_0^a |q(t^2)|\,dt.
\]
Combining the two bounds, we obtain the lower bound
\begin{equation}\label{eq:small_interval_sup_ratio}
 \frac{\sup_{t\in[b^2,a^2]} |q(t)|}{\sup_{t\in[0,a^2]} |q(t)|} \geq \frac{a}{(a-b)2n(2n-1)H}.
\end{equation}

We now proceed to derive an upper bound for the same norm ratio in terms of the number of roots of $p(x)$ in the interval $[b^2,1)$.
\begin{lemma}
    Let $f(x)$ be a polynomial of degree $D$ with $r\in[1,D]$ roots in the interval $[\beta,\alpha]$ where $0<\beta<\alpha$. Then we have the bound
    \[
    \frac{\sup_{x\in[\beta,\alpha]}|f(x)|}{\sup_{x\in[0,\alpha]}|f(x)|} \leq \left(\frac{e^2D^2(1-\beta/\alpha)}{r^2}\right)^r.
    \]
\end{lemma}
\begin{proof}
    We compare both norms against the maximum of $|f^{(r)}(x)|$ on the interval $[0,\alpha]$. 

    On one hand, for any $x\in[\beta,\alpha]$, we have
    \[
    |f(x)| = \left|\frac{f^{(r)}(\xi)}{r!} \prod_{i=1}^r (x-x_i)\right| \leq \frac{(\alpha-\beta)^r}{r!}\sup_{x\in[0,\alpha]} |f^{(r)}(x)| ,
    \]
    where $x_1,\dots,x_r$ are the roots of $f$ in $[\beta,\alpha]$ and the existence of $\xi\in[\beta,\alpha]$ is guaranteed by the interpolation remainder formula.

    On the other hand, the classical Markov brothers' inequality (see, e.g., \cite{rivlin1990chebyshev}) implies that 
    \[
    \sup_{x\in[0,\alpha]} |f^{(r)}(x)| \leq \frac{D^{2r}}{(\alpha/2)^r(2r-1)!!} \sup_{x\in[0,\alpha]} |f(x)|.
    \]
    Combining this with the previous bound gives
    \[
    \frac{\sup_{x\in[\beta,\alpha]}|f(x)|}{\sup_{x\in[0,\alpha]}|f(x)|} \leq \frac{1}{(2r)!} \left(4D^{2}(1-\beta/\alpha)\right)^r \leq \left(\frac{e^2D^2(1-\beta/\alpha)}{r^2}\right)^r.
    \]

    This completes the proof.
\end{proof}

Applying the lemma with $f(x)=q(x)$, $D=n-1$, $\beta=b^2$, and $\alpha=a^2$, we obtain
\begin{equation}\label{eq:small_interval_sup_ratio_2}
\frac{\sup_{t\in[b^2,a^2]} |q(t)|}{\sup_{t\in[0,a^2]} |q(t)|} \leq \left(\frac{e^2(n-1)^2(1-b^2/a^2)}{r^2}\right)^r,
\end{equation}
where $r=N_b-1$.

Combining the lower bound \eqref{eq:small_interval_sup_ratio} and the upper bound \eqref{eq:small_interval_sup_ratio_2}, we obtain 
\begin{equation}\label{eq:combined_sup_ratio}
\frac{a}{(a-b)2n(2n-1)H} \leq \left(\frac{e^2(n-1)^2(1-b^2/a^2)}{r^2}\right)^r.
\end{equation}
To extract a bound for $r$, we note that if $r\geq \sqrt2en\sqrt{1-b^2/a^2}$, the right-hand side of \eqref{eq:combined_sup_ratio} is at most $2^{-r}$, which implies
\[
    r \leq \log_2\frac{(a-b)2n(2n-1)H}{a}< \log_2(4n^2H).
\]
Therefore, in all cases, we have
\[
    N_b=1+r \leq 1+\max\left(\sqrt2\,en\sqrt{1-b^2/a^2}, \log_2(4n^2H)\right).
\]
The same bound also holds trivially when $N_b\leq1$.

\subsection{Parameter selection}
\begin{proof}[Proof of Theorem~\ref{thm:moment_lower_bound}]
Let $L:=\log(2n)$, and fix constants $2<\beta<\gamma$.
We choose
\[
    \tau:=L-\gamma\log L,\qquad b:=1-\frac{L^4}{4n^2},\qquad H:=\left\lceil\frac{2n}{L^\beta}\right\rceil,\qquad d:=\left\lfloor\frac{2n}{L}\right\rfloor.
\]
For sufficiently large $n$, these choices satisfy all the required parameter restrictions. We verify that they give $\lambda_{a,n}(B)=o(1)$ and $Q=(1+o(1))n$.

Since $a=\cosh(\tau/(2n))$ and $d\sim 2n/L$, the Legendre bound \eqref{eq:legendre_bound_2} gives
\[
    P_{2(n+d)+1}(a)-1 \leq \exp\left(\frac{2(n+d)+1}{2n}\tau\right) =\exp(\tau+2+o(1))=O\left(\frac{n}{L^\gamma}\right).
\]
Moreover,
\[
    d\,\operatorname{arcosh}(1/b) \sim d\sqrt{2(1-b)} \sim \sqrt{2}\,L.
\]
Applying \eqref{eq:chebyshev_hyperbolic} with $u=\operatorname{arcosh}(1/b)$, we obtain
\[
    \frac{1}{T_d(1/b)^2} \leq 4\exp\left(-2d\,\operatorname{arcosh}(1/b)\right) =O(n^{-1}).
\]
Together with $H^{-1}=O(L^\beta/n)$, \eqref{eq:lambda_upper_bound} therefore implies
\[
    \lambda_{a,n}(B)=O(L^{\beta-\gamma})=o(1).
\]
In particular, $\lambda_{a,n}(B)<1$ for sufficiently large
$n$, so the construction above applies.

We next estimate the number $N_b$ of positive roots of $t\mapsto p(t^2)$ in the high-weight interval $[b,1)$. Since
\[
    a^2-1=O\left(\frac{L^2}{n^2}\right), \qquad 1-b^2=O\left(\frac{L^4}{n^2}\right),
\]
we have
\[
    1-\frac{b^2}{a^2} =O\left(\frac{L^4}{n^2}\right).
\]
The root-count bound proved above now gives
\[
    N_b \leq 1+\max\left(\sqrt{2}\,en\sqrt{1-b^2/a^2},\log_2(4n^2H) \right)=O(L^2).
\]
Consequently,
\[
    (H-1)(N_b+1)=O(nL^{2-\beta})=o(n).
\]

To estimate the remaining root-sum error, observe that
\[
    b+H(1-b)=1+(H-1)(1-b)=1+O\left(\frac{L^{4-\beta}}{n}\right)=1+o(1).
\]
Since $\lambda_{a,n}(B)=o(1)$, the definition \eqref{eq:M_definition} therefore gives
\[
    M=(1+o(1))4n^2,\qquad\log M=2L+o(1).
\]
Combining \eqref{eq:Q_bounds} and \eqref{eq:root_sum_bound},
we obtain
\[
    n-\log M \leq Q \leq n+\log M+a^2+2(H-1)(N_b+1),
\]
and hence
\[
    Q=n+O\left(L+nL^{2-\beta}\right)=(1+o(1))n.
\]

Finally, $\tau\sim L$ and $a=\cosh(\tau/(2n))$ imply
\[
    a(a-1)\sim\frac{\tau^2}{8n^2} \sim\frac{L^2}{8n^2}.
\]
Since $r_0=a$, $r_{\min}\geq-1$, and $g=2n+1$,
the constructed list satisfies
\[
    \frac{r_0(r_0+r_{\min})}{Q} \geq\frac{a(a-1)}{Q} =(1+o(1))\frac{L^2}{8n^3} =(1+o(1))\frac{(\log g)^2}{g^3}.
\]
This proves the desired asymptotic lower bound. For a concrete choice of the logarithmic exponents,
one may take $\beta=3$ and $\gamma=4$.
\end{proof}

\section{Regular Graphs}
\label{sec:regular_graphs}

In this section, we use the local structure of regular graphs to prove an upper bound of order $g^{-3}$ for $\gamma_g^{\mathrm{reg}}$. The proof uses test vectors with alternating signs supported on bipartite balls. We bound their energy edge by edge and average their squared norms over the vertices of a shortest odd cycle to obtain the required Rayleigh quotient estimate.

We write $N(v)$ for the open neighborhood of a vertex $v$, and $\dist(v,w)$ for graph distance. For an integer $j\geq0$, the closed ball $B_j(v)$ is the set of vertices at distance at most $j$ from $v$. The norm $\|q\|$ of a vector is its Euclidean norm.

\begin{theorem}
For every $d$-regular graph $G$ on $n$ vertices of odd girth exactly $g\geq5$, we have
\[
    \frac{d+\lambda_n(G)}{n} \leq \frac{12}{(g-3)(g-2)(g-1)}.
\]
\end{theorem}
\begin{proof}
    Let $r=(g-3)/2$. For every vertex $u\in V(G)$, the induced graph on the ball $B_r(u)$ is bipartite (otherwise one can construct an odd cycle of length at most $2r+1<g$ inside the ball). We construct a ``cone vector'' supported on $B_r(u)$, and use its Rayleigh quotient to bound $\lambda_n(G)$. Define
    \[
    q_u(v):=\begin{cases}(-1)^{\dist(u,v)}(r+1-\dist(u,v)),&\dist(u,v)\leq r,\\0,&\dist(u,v)>r,\end{cases}.
    \]

    We first prove an upper bound on the quantity $q_u^\top (dI+A) q_u$, where $A$ is the adjacency matrix of $G$.

    \textbf{Claim.} For every vertex $u\in V(G)$, we have
    \[
        q_u^\top (dI+A) q_u \leq nd/2.
    \]
    \begin{proof}[Proof of Claim]
        For every edge $\{v,w\}\in E(G)$, we have $|q_u(v)+q_u(w)|\leq1$. Indeed, if both endpoints lie outside $B_r(u)$, both values are zero. If exactly one endpoint lies outside $B_r(u)$, the other has distance exactly $r$ from $u$ and hence has $q_u$-value $\pm1$. Finally, if both endpoints lie in $B_r(u)$, their distances from $u$ differ by exactly one, since they are adjacent and $B_r(u)$ is bipartite. Thus their $q_u$-values have opposite signs and magnitudes differing by one, giving $|q_u(v)+q_u(w)|=1$. This in turn implies that
        \[
        q_u^\top (dI+A) q_u = \sum_{\{v,w\}\in E(G)} (q_u(v)+q_u(w))^2 \leq |E(G)| = nd/2.
        \]
    \end{proof}

    This gives control over the numerator in the Rayleigh quotient for $q_u$. To bound the denominator, we average over the vertices of a shortest odd cycle in the graph. Let $C$ be a shortest odd cycle in $G$, and identify $C$ with its vertex set when writing sums, intersections, or cardinalities. Thus $|C|=g$.
    \textbf{Claim.} We have the lower bound
    \[
        \frac1g\sum_{u\in C} \|q_u\|^2 \geq dr(r+1)(2r+1)/6.
    \]
    \begin{proof}[Proof of Claim]
        Since $C$ is a shortest odd cycle, it is induced. Moreover, if $v\notin C$ has two distinct neighbors $x,y\in C$, let $\ell$ be the length of the odd $x$--$y$ arc of $C$. Together with the path $xvy$, this arc forms an odd cycle of length $\ell+2$. Hence $\ell+2\geq g$, forcing $\ell=g-2$. Thus $x$ and $y$ are at cyclic distance two. No three vertices of an odd cycle of length at least five are pairwise at cyclic distance two, so every vertex has at most two neighbors in $C$.

        Let $B_1(C)$ denote the set of vertices at distance at most 1 from the cycle $C$. We separate the vertices $v\in B_1(C)$ into two cases:
        \begin{itemize}
        \item If $|N(v)\cap C|=1$, then $|B_j(v)\cap C|\geq 2j-1\geq j$ for all $1\leq j\leq r$.
        \item If $|N(v)\cap C|=2$, then the two neighbors must be at distance 2 along the cycle, both for $v\in C$ and $v\in B_1(C)\setminus C$. In this case, $|B_j(v)\cap C|\geq2j$ for all $1\leq j\leq r$.
        \end{itemize}
        In both cases we have the lower bound $|B_j(v)\cap C|\geq j|N(v)\cap C|$ for all $1\leq j\leq r$.

        Therefore, we have
        \begin{align*}
        \sum_{u\in C} \|q_u\|^2 &\geq \sum_{v\in B_1(C)} \sum_{u\in C} q_u(v)^2 \\
        &= \sum_{v\in B_1(C)} \sum_{j=0}^{r} \sum_{u \in C, \dist(u,v)=j} (r+1-j)^2 \\
        &= \sum_{v\in B_1(C)} \sum_{j=0}^{r} \sum_{u \in C, \dist(u,v)=j} \sum_{k=j}^{r} (2(r+1-k)-1) \\
        &= \sum_{v\in B_1(C)} \sum_{k=0}^{r} \sum_{u \in C, \dist(u,v)\leq k} (2(r+1-k)-1) \\
        &\geq \sum_{k=1}^{r} \sum_{v\in B_1(C)} |B_k(v)\cap C| (2(r+1-k)-1) \\
        &\geq \left(\sum_{k=1}^{r} k(2(r+1-k)-1)\right) \left( \sum_{v\in B_1(C)} |N(v)\cap C| \right) \\
        &=gd\,r(r+1)(2r+1)/6.
        \end{align*}
    \end{proof}

    With these two claims, we can conclude that there exists a vertex $u\in C$ such that
    \[
        \frac{q_u^\top (dI+A) q_u}{\|q_u\|^2} \leq \frac{nd/2}{dr(r+1)(2r+1)/6} = \frac{3n}{r(r+1)(2r+1)}.
    \]

    And therefore
    \[
        \frac{d+\lambda_n(G)}{n} \leq \frac{3}{r(r+1)(2r+1)}=\frac{12}{(g-3)(g-2)(g-1)}.
    \]
\end{proof}


\section{Conclusion}
\label{sec:conclusion}

We have determined the asymptotic order of the abstract moment problem, proving that
\[
    (1-o(1))\frac{(\log g)^2}{g^3} \leq \gamma'_g \leq \left(\frac{10}{7}+o(1)\right)\frac{(\log g)^2}{g^3}.
\]
For regular graphs, we have shown that $\gamma_g^{\mathrm{reg}}=\Theta(g^{-3})$. Thus the moment relaxation permits values larger by a factor of order $(\log g)^2$ than those attainable by regular graphs.

The asymptotic order of $\gamma_g$ for general graphs remains open. In particular, does the bound $\gamma_g=O(g^{-3})$ hold without the regularity assumption? Our lower bound for $\gamma'_g$ shows that such a result would require information beyond the vanishing odd moments and the quadratic constraint.

We conjecture that the lower bound for the abstract problem has the sharp leading constant, namely
\[
    \gamma'_g=(1+o(1))\frac{(\log g)^2}{g^3}.
\]
For regular graphs, we conjecture that $g^3\gamma_g^{\mathrm{reg}}\to\pi^2$ as $g\to\infty$ through odd integers, so that odd cycles are asymptotically extremal.

\textbf{AI Disclosure.} During the development of this work, the author used OpenAI GPT-6 Astra (accessed September--October 2026) for exploratory numerical and symbolic computations, code generation, testing possible reductions and conjectures, discussing proof strategies for intermediate results, and drafting and revising portions of the manuscript. In particular, the development of the asymptotic upper-bound argument benefited from extended iterative discussion with the model. The research questions and overall direction were determined by the author, who evaluated the computational and mathematical suggestions and independently reconstructed and verified all arguments appearing in the final manuscript. The author wrote and revised the manuscript, incorporating selected model-generated drafts and suggestions, and determined its final content, organization, and presentation. The author takes full responsibility for all statements and any errors in the paper. 

\printbibliography

\end{document}